\documentclass{amsart}
\usepackage{amsmath,amsfonts,amssymb,amsthm}
\newcommand{\va}{l}
\newcommand{\Deg}{N}
\newcommand{\noz}{K}
\newcommand{\Tree}{\operatorname{T}}

\renewcommand{\epsilon}{\varepsilon}
\newcommand{\Vol}{\operatorname{Vol}}

\newcommand{\area}{\operatorname{area}}

\newcommand\N{\mathbb N}

\newcommand\Z{\mathbb Z}

\newcommand{\T}{\mathbb T}

\newcommand{\R}{{\mathbb R}}
\newcommand{\reals}{{\mathbb R}}

\newcommand\C[1]{{\mathbb C}^{#1}}
\newcommand{\cx}{{\mathbb C}}

\newcommand{\CP}{{\mathbb C}\!\operatorname{P}^1}

\newcommand{\cH}{{\mathcal H}}

\newcommand{\cQ}{{\mathcal Q}}
\newcommand{\cP}{\mathcal P}

\newcommand{\cZ}{\mathcal{Z}}

\newlength{\halfbls}

\def\@ifundefined#1#2#3%
  {\expandafter\ifx\csname#1\endcsname\relax#2\else#3\fi}

\@ifundefined{theoremstyle}{
}{
\theoremstyle{plain} 
}
\newtheorem{theorem}{Theorem}[section]
\newtheorem*{NoNumberTheorem}{Theorem}

\newtheorem{lemma}[theorem]{Lemma}

\newtheorem{corollary}[theorem]{Corollary}

\@ifundefined{theoremstyle}{
}{
\theoremstyle{definition} 
}
\newtheorem{definition}[theorem]{Definition}

\newtheorem{Convention}[theorem]{Convention}

\newtheorem{remark}[theorem]{Remark}

\numberwithin{equation}{section}

\newlength{\figboxwidth}
\begin{document}

\title{Counting generalized Jenkins--Strebel differentials}

\date{May 21, 2013}

\begin{abstract}
We  study  the combinatorial geometry of ``lattice'' Jenkins--Strebel
differentials with simple zeroes and simple poles on $\CP$ and of the
corresponding   counting   functions.   Developing   the  results  of
M.~Kontsevich~\cite{Kontsevich}  we   evaluate  the  leading  term  of  the  symmetric
polynomial  counting  the number of such ``lattice'' Jenkins--Strebel
differentials  having  all  zeroes  on  a single singular layer. This
allows   us   to   express   the   number   of   general  ``lattice''
Jenkins--Strebel  differentials  as  an appropriate weighted sum over
decorated trees.

The  problem of counting Jenkins--Strebel differentials is equivalent
to  the problem of counting pillowcase covers, which serve as integer
points  in  appropriate local coordinates on strata of moduli spaces of meromorphic quadratic differentials.
This  allows  us  to  relate  our  counting  problem  to calculations of volumes of these strata . A very explicit expression for the volume of any stratum
of  meromorphic  quadratic  differentials  recently  obtained  by the
authors~\cite{AEZ:pillowcase:covers} leads to an interesting combinatorial identity for our
sums over trees.
\end{abstract}

\author{Jayadev~S.~Athreya}
\address{
Department of Mathematics,
University of Illinois,
Urbana, IL~61801 USA}
\email{jathreya@illinois.edu}

\author{Alex Eskin}
\address{
Department of Mathematics,
University of Chicago,
Chicago, Illinois 60637, USA\\
}
\email{eskin@math.uchicago.edu}

\author{Anton Zorich}
\address{
Institut de math\'ematiques de Jussieu,
Institut Universitaire de France,
Universit\'e Paris 7, France}
\email{zorich@math.jussieu.fr}

\thanks{
J.S.A. is partially supported by NSF grants DMS 0603636
and DMS 0244542.
A.E. is partially supported by NSF grants DMS 0244542 and DMS 0604251.
A.Z. is partially supported by the program PICS of CNRS and by ANR ``GeoDyM''
}

\maketitle

\setcounter{tocdepth}{2} 
\tableofcontents


\section{Introduction}
\label{intro}

\subsection{Counting pillowcase covers}
\label{subsec:counting:pillowcase:covers}

A  geometric  approach  to  volume  computation for the strata in the
moduli  spaces  of  Abelian  or  quadratic  differentials consists in
counting   \textit{square-tiled   surfaces}   or   \textit{pillowcase
covers},            see~\cite{Eskin:Okounkov},            \cite{EO2},
\cite{Eskin:Okounkov:Pandharipande}, \cite{Zorich:volumes}.
A  pillowcase  cover  $\hat\cP\to\CP$  is a ramified cover over $\CP$
branched  over  four  points. Define a flat metric on $\CP$ such that
the      resulting     \textit{pillowcase     orbifold}     as     in
Figure~\ref{fig:square:pillow}
\begin{figure}[htb]
   %
   %
\includegraphics{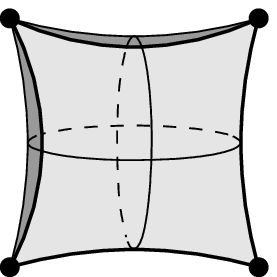}
\vspace{20pt}
\caption{
\label{fig:square:pillow}
Pillowcase orbifold $\cP$.}
\end{figure}
is glued from two squares of size $1/2 \times 1/2$. Choosing the four
corners  of  the pillowcase $\cP$ as the four ramification points, we
get   an   induced   square  tiling  of  the  pillowcase  cover,  see
Figure~\ref{fig:global:graph:and:layers}.
\begin{figure}[htb]
\includegraphics{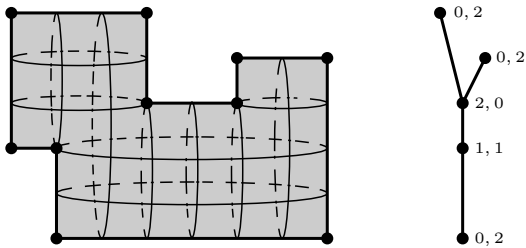}
\begin{picture}(0,0)(0,0)
\put(75,-6){\tiny $0,2$}
\put(91,-23){\tiny $0,2$}
\put(83,-40.5){\tiny $2,0$}
\put(83,-57.5){\tiny $1,1$}
\put(83,-91.5){\tiny $0,2$}
\end{picture}
\vspace{92pt}
\caption{
\label{fig:global:graph:and:layers}
Pillowcase  cover  $\hat\cP$  and  associated  decorated tree $\Tree$.
}
\end{figure}
The  flat  structure  on  the  pillowcase  $\cP$  corresponds  to the
meromorphic  quadratic  differential $\psi_0=(dz)^2$ on $\cP=\T/\pm$,
where  $\T=\C{}/(\Z\oplus  i\Z)$. The quadratic differential $\psi_0$
has  four  simple  poles  at  the  corners of the pillow and no other
singularities.  We  shall see that the induced quadratic differential
$\psi=\pi^\ast  \psi_0$  on $\hat\cP$ defines an integer point (in appropriate local coordinates) in the
ambient stratum of meromorphic quadratic differentials.

In  this paper we want to count the number of nonisomorphic connected
pillowcase  covers  $\hat\cP$  of  degree  at  most $\Deg$ having the
following  ramification  pattern. All ramification points are located
over  the corners of the pillowcase. All preimages of the corners are
ramification points of degree two with exception for $K$ ramification
points  of degree three and for $K+4$ unramified points. For example,
the  pillowcase cover in Figure~\ref{fig:global:graph:and:layers} has
$K=3$  ramification  points  of  degree  three and $K+4=7$ unramified
points.   We   do  not  specify  how  the  projections  of  $K+(K+4)$
distinguished  points are distributed between the four corners of the
pillowcase $\cP$.

Our  restriction  on the ramification data implies that the quadratic
differential  $\psi$  has  exactly  $K$ simple zeroes (located at the
ramification  points  of  degree  three); it has exactly $K+4$ simple
poles (located at $K+4$ nonramified preimages of the corners), and it
has  no  other  zeroes  or poles. In particular, the pillowcase cover
$\hat\cP$ has genus zero.

In  order  to  count pillowcase covers we note that if $\hat\cP$ is a
pillowcase cover, it can be decomposed into horizontal cylinders with
integer widths, with zeros and poles lying on the boundaries of these
cylinders,   see  Figure~\ref{fig:global:graph:and:layers}.  We  call
these   boundaries  \textit{singular  layers}.  Each  singular  layer
defines  a  connected  graph  with  a certain  number  $m$ of trivalent
vertices,  a certain  number  $n$  of  univalent  vertices, and with no
vertices  of  any  other  valence. Actually, it is more convenient to
consider  the  singular  layer as a \textit{ribbon graph} by taking a
small  tubular neighborhood of the singular layer inside the surface.
The graph is \textit{metric}: all edges have certain lengths measured
by  means of the flat structure. Since the length of the sides of each square of the
tiling  of  the  pillowcase  cover is $1/2$, and the vertices of each
singular  layer  are  located  at  the  vertices  of the squares, the
lengths  of  all  edges  of our graph are half-integer. Note that the
number  $l$ of cylinders adjacent to a layer is expressed in terms of
the  number  $m$  of  zeroes  and  the  number  $n$  of  poles on the
corresponding  layer  as $l=\frac{(m-n)}{2} +2$. Thus, by topological
reasons $m-n+2$ is necessarily a nonnegative even number.
\begin{figure}[htb]
\includegraphics{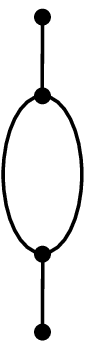}
\includegraphics{Gamma_1.eps}
\includegraphics{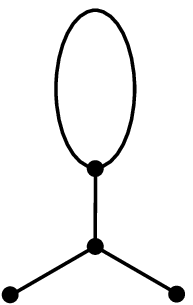}
\includegraphics{Gamma_2.eps}
\includegraphics{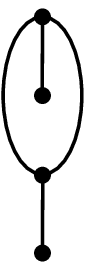}
\begin{picture}(0,0)(0,0)
\begin{picture}(0,0)(0,0)
\put(-170,-98){$\Gamma_{1,1}$}
\put(-148,-54.5){$1$}
\put(-166,-54.5){$l_1$}
\put(-132.5,-54.5){$l_2$}
\put(-144.5,-21){$l_3$}
\put(-144.5,-88){$l_4$}
\put(-166,-21){$2$}
\end{picture}
\begin{picture}(0,0)(-62,0)
\put(-170,-98){$\Gamma_{1,2}$}
\put(-148,-54.5){$2$}
\put(-166,-21){$1$}
\end{picture}
\begin{picture}(0,0)(-141,0)
\put(-166,-98){$\Gamma_{2,1}$}
\put(-160.5,-41){$1$}
\put(-176,-21){$2$}
\end{picture}
\begin{picture}(0,0)(-228,0)
\put(-166,-98){$\Gamma_{2,2}$}
\put(-160.5,-41){$2$}
\put(-176,-21){$1$}
\end{picture}
\begin{picture}(0,0)(-288,0)
\put(-166,-98){$\Gamma_3$}
\put(-149,-66){$1$}
\put(-161,-28){$2$}
\end{picture}
\end{picture}
\vspace{100pt}
\caption{
\label{fig:graph1}
The  list  of  all  connected  ribbon  graphs  with labelled boundary
components  having  $m$ zeros (i.e. $m=2$ trivalent vertices) and $2$
poles (i.e. $n=2$ univalent vertices).
}
\end{figure}

Developing  the techniques of M.~Kontsevich from~\cite{Kontsevich}, we
find  a  formula  for  the  following  counting  function.  Given $l$
positive  integer  numbers  $w_1,  \dots, w_l$ we count the number of
ways  to  join  $l$ cylinders of widths $w_1, \dots, w_l$ together by
means  of  a connected half-integer ribbon graph having $m$ trivalent
and   $n$   univalent   vertices;   see  Figure~\ref{fig:graph1}.

\begin{theorem}
\label{theorem:local:poly}
Let   $m$   and   $n$   be  nonnegative  integer  numbers  not  equal
simultaneously  to  zero  such  that  $m-n+2$  is  a nonnegative even
number.  Let  $F_{m,n}(w_1,\dots,w_l)$, where $l=\frac{(m-n)}{2} +2$,
be  the  number  of  ways  to  attach $l$ cylinders of integer widths
$w_1,\dots,w_l$  to all possible layers containing $m$ zeroes and $n$
poles,  in  such  way  that  all  edges  of  the  resulting graph are
half-integer. Up to the lower order terms one has
\begin{equation}
\label{eq:local:poly}
F_{m,n}(w_1,\dots,w_l) =
\frac{m!}{\left(\frac{(m+n)}{2} -1)\right)!}
\sum_{\substack{b_1,\ldots b_l\\ \sum b_i = \frac{(m+n)}{2} -1}}
\binom{\frac{(m+n)}{2} -1}{b_1, \ldots b_l}^2 \cdot
\prod_{i=1}^l w_i^{2b_i}\,,
\end{equation}
\end{theorem}
Theorem~\ref{eq:local:poly} is of independent interest; it is
is  proved  in  \S\ref{subsec:recurrense}. To elaborate certain
geometric   intuition  helpful  in  manipulating  geometric  counting
functions  we compute in \S\ref{subsubsec:example} by hands the
function  $F_{2,2}(w_1,w_2)$  corresponding  to  ribbon  graphs  from
Figure~\ref{fig:graph1}.

Having  studied  the  enumerative  geometry of singular layers let us
return  to  global pillowcase covers. Suppose there are $k$ cylinders
of  width $w_i$ and height $h_i$ respectively. Since the flat surface
is  a  topological  sphere,  there  are  $k+1$ singular layers in the
decomposition  of $\hat\cP$. The total number of pillowcase covers of
degree $\Deg$ with this type of decomposition can be written as
\begin{equation}
\label{eq:eqn:k:layer}
2^k
\sum_{\substack{w\cdot h \le\Deg\\
w_i,h_i\in\N}}
w_1\cdot w_2\dots \cdot w_k\cdot
\prod_{i=1}^{k+1} F_{m_i,n_i}(w_1, \ldots, w_k)\,,
\end{equation}
where  $F_{m_i,n_i}$  is  a  function counting the number of ways the
cylinders  of  width $w_i$ can be glued at the layer $i$, and $w\cdot
h:=\sum_{i=1}^k  w_i  h_i$.  The factor $(2w_1) (2w_2) \ldots (2w_k)$
arises  from  the  possibility  of  twisting each cylinder around the
waist curve; see \S\ref{sec:subsec:jenkins} for more details.

Representing  each  singular layer by a vertex of an associated graph
$\Tree$  as  in  Figure~\ref{fig:global:graph:and:layers},  and every
cylinder  by  an  edge  of such graph, we encode the decomposition of
$\hat\cP$  into  cylinders  by a global graph $\Tree$. We also record
the information on the number $m_i$ of zeroes and the number $n_i$ of
simple  poles  located  at  each  layer  $i$. This extra structure is
referred   to   as   a  \textit{decoration}.  Since  $\hat\cP$  is  a
topological   sphere,   the  graph  $\Tree$  is  a  tree.  Taking  an
appropriate   sum   of  expressions~\eqref{eq:eqn:k:layer}  over  all
decorated  trees  we  get the leading term of the asymptotics for the
number  of  pillowcase covers (see \S\ref{sec:subsec:total} and
Theorem~\ref{th:volume:through:pillows} for exact statements).

On   the   other   hand,   we   have   the  following  recent  result
from~\cite{AEZ:pillowcase:covers}:
\begin{NoNumberTheorem}
$$
\Vol\cQ_1(1^K,-1^{K+4})=\frac{\pi^{2K+2}}{2^{K-1}}\,.
$$
\end{NoNumberTheorem}
It    implies    the    main   combinatorial   identity   stated   in
Theorem~\ref{th:volume:through:pillows}.
\bigskip

The  formula for the volume $\Vol\cQ_1(1^K,-1^{K+4})$ (and, actually,
a  much  more  general  formula  for  the  volume  of  any stratum of
meromorphic  quadratic  differentials  with  at most simple poles) is
obtained  in~\cite{AEZ:pillowcase:covers}  in  a  very  indirect  way
through the analytic Riemann-Roch theorem, asymptotics of the determinant
of  the  Laplacian of the singular flat metric, principal boundary of
the moduli spaces, Siegel--Veech constants, and Lyapunov exponents of
the  Hodge bundle. The current paper develops a transparent geometric
approach.  We  have to admit that from purely pragmatic point of view
this natural geometric approach is, however, less efficient.

The   situation   with   counting   volumes   of  strata  of  Abelian
differentials  is somehow similar: the problem was solved by A.~Eskin
and    A.~Okounkov    in~\cite{Eskin:Okounkov}   using   methods   of
representation  theory  of the symmetric group, and developed further
in~\cite{EO2}    and   in~\cite{Eskin:Okounkov:Pandharipande}   using
techniques  of  quasimodular  forms.  A  straightforward  counting of
square-tiled  surfaces  works only for strata of small dimension, and
becomes   disastrously   complicated   when   the   dimension  grows,
see~\cite{Zorich:volumes}.  However, the technique elaborated in this
naive  geometric  approach to the study of square-tiled surfaces, and
the  ties to various related subjects proved to be extremely helpful.
For  example,  the  separatrix  diagrams  (analogs  of  ribbon graphs
representing   singular   layers)  were  used  as  one  of  the  main
instruments  in  classification~\cite{Kontsevich:Zorich} of connected
components  of  the  strata.  Multiple  zeta  values  which appear in
counting  square-tiled  surfaces  represented  by  certain  groups of
separatrix   diagrams,  seem  to  have  interesting  applications  to
representation theory.

We  believe  that an ample description of the enumerative geometry of
pillowcase covers combining direct geometric approach elaborated
in  the current paper, and the implicit analytic approach
from~\cite{AEZ:pillowcase:covers}   could   be  helpful  for  various
applications.
\subsection{Reader's guide}

In   \S\ref{sec:Elementary:geometry}   we   present  the  basic
background  material  on the natural volume element the moduli spaces
of             quadratic            differentials.            Namely,
in~\S\ref{ss:coordiantes:in:stratum}   we   introduce  the  canonical
\textit{cohomological      coordinates}      in      each     stratum
$\cQ(d_1,\dots,d_k)$  of  meromorphic quadratic differentials with at
most  simple  poles.  In~\S\ref{sec:subsec:notmalization} we define a
canonical  lattice  in these coordinates which determines the natural
linear        volume        element       in       the       stratum.
In~\S\ref{subsec:integer:points}  we  show  how volume calculation is
related to counting of lattice points.

The    original    part    of    the    paper    is    presented   in
\S\ref{sec:Jenkins}.
In~\S\ref{subsec:lattice:square:tiled:pillowcase} we show why lattice
points  in the stratum are represented by pillowcase covers which, in
view   of~\S\ref{subsec:integer:points},   explains  why  the  volume
calculation   is   equivalent  to  counting  the  pillowcase  covers.
In~\S\ref{sec:subsec:jenkins}   we   discuss   in  more  details  the
functions   $F_{n,m}$  from  Theorem~\ref{theorem:local:poly},  study
their elementary properties and prove formula~\eqref{eq:eqn:k:layer}.
We  consider in \S\ref{subsubsec:example} a particular case $F_{2,2}$
corresponding  to Figure~\ref{fig:graph1} as an example, for which we
perform       an       explicit       by      hand     computation.
In~\S\ref{subsubsec:kontsevich}  we  obtain  a general expression for
$F_{m,0}$  as  a  corollary from Kontsevich's Theorem~\cite{Kontsevich}. We use this
expression      as     a     base     of     recurrence     developed
in~\S\ref{subsec:recurrense}, where we express $F_{m+1,n+1}$ in terms
of    $F_{m,n}$.    This    recurrence    allows    us    to    prove
in~\S\ref{subsec:recurrense}        Theorem~\ref{theorem:local:poly}.
Finally,  in~\S\ref{sec:subsec:total}  we  compute  the  sum over all
decorated   trees   and   prove   the   main   identity   stated   in
Theorem~\ref{th:volume:through:pillows}.  We  illustrate this Theorem
performing  a  detailed  computation for the strata $\cQ(1,-1^5)$ and
$\cQ(1^2,-1^6)$   in~\S\ref{subsec:1:5}   and   in~\S\ref{subsec:2:6}
respectively.

\medskip
\noindent\textbf{Acknowledgments.}
The  authors  are  happy  to  thank  IHES,  IMJ,  IUF,  MPIM, and the
Universities   of   Chicago,  of  Illinois  at  Urbana-Champaign,  of
Rennes~1,  and  of  Paris~7 for hospitality during the preparation of
this paper. We thank the anonymous referee for their careful  reading
of the paper and helpful suggestions.

\section{Canonical volume element in the  moduli  space  of  quadratic
differentials}
\label{sec:Elementary:geometry}

\subsection{Coordinates in a stratum of quadratic differentials.}
\label{ss:coordiantes:in:stratum}
Consider a meromorphic quadratic differential $\psi$ having zeroes of
arbitrary  multiplicities  but  only simple poles on $\CP$. Let $P_1,
\ldots,  P_n$  be  its  singular  points  (zeros  and  simple poles).
Consider  the minimal branched double covering $p:\hat S\to \CP$ such
that   the   induced   quadratic  differential  $p^\ast\psi$  on  the
hyperelliptic  surface  $\hat  S$  is  already a square of an Abelian
differential $p^\ast\psi=\omega^2$.
\begin{figure}[htb]
%
%
%
\includegraphics{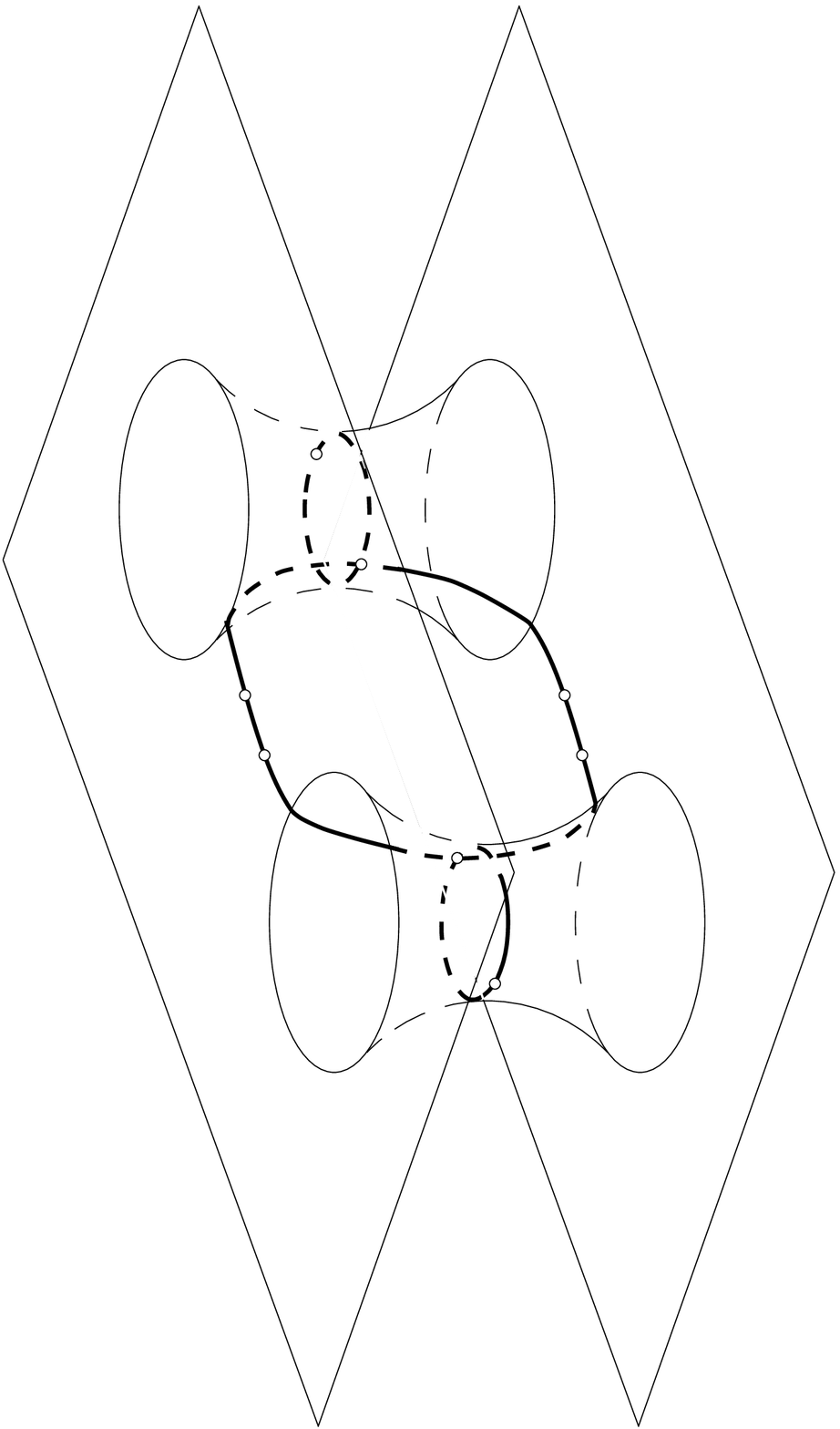}
\includegraphics{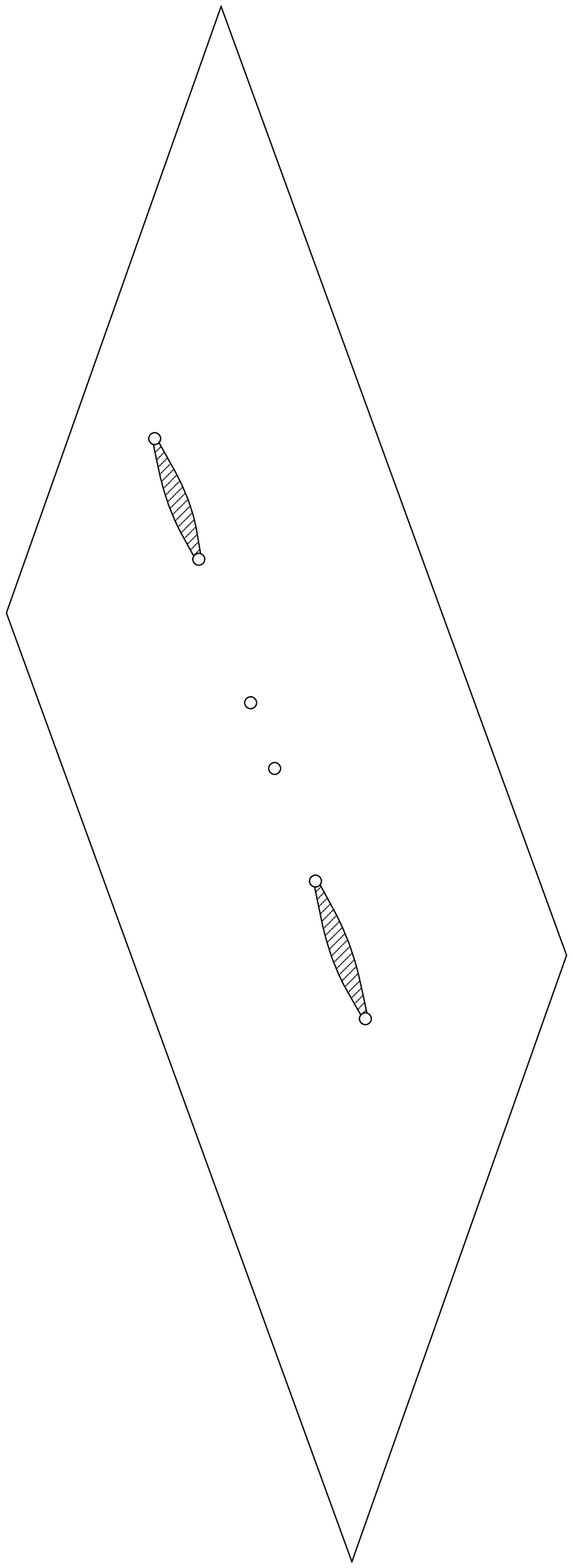}
\begin{picture}(0,0)(0,0)
\put(-53,-101){$\hat P_1$}
\put(-17,-88){$\hat P_2$}
\put(0,-58){$\hat P^+_i$}
\put(0,-118){$\hat P^-_i$}
\put(35,-65){$\hat P_{n-1}$}
\put(61,-59){$\hat P_n$}
\end{picture}
\begin{picture}(0,0)(0,0)
\put(-48,-223){$P_1$}
\put(-15,-213){$P_2$}
\put(3,-206){$P_i$}
\put(33,-194){$P_{n-1}$}
\put(65,-184){$P_n$}
\end{picture}
\vspace{240bp}
\caption{
\label{fig:hyperel}
Basis  of  cycles  in $H_1^-(\hat S,\{\hat P_1,\dots,\hat P_N\};\Z)$.
Note that the cycle corresponding to the very last slit is omitted.
}
\end{figure}

The  zeros  $\hat{P}_1,\ldots,\hat{P}_N$  of  the  resulting  Abelian
differential  $\omega$  correspond  to  the  zeros  of  $\psi$ in the
following  way:  every  zero  $P\in  \CP$ of $\psi$ of odd order is a
ramification  point  of  the  covering,  so it produces a single zero
$\hat{P}\in\hat  S$  of  $\omega$;  every zero $P\in\CP$ of $\psi$ of
even  order  is  a  regular point of the covering, so it produces two
zeros  $\hat{P}^+,\hat{P}^-\in\hat  S$ of $\omega$. Every simple pole
of  $\psi$ defines a branching point of the covering; this point is a
regular point of $\omega$.

Consider  the subspace $H_1^-(\hat S,\{\hat P_1,\dots,\hat P_N\};\Z)$
of  the relative homology of the cover with respect to the collection
of   zeroes  $\{\hat  P_1,\dots,\hat  P_N\}$  of  $\omega$  which  is
antiinvariant with respect to the induced action of the hyperelliptic
involution.  We  are  going to construct a basis in this subspace (in
complete  analogy  with  a  usual  basis  of  absolute  cycles  for a
hyperelliptic surface).

We  can  always  enumerate  the  singular  points $P_1,\dots, P_n$ of
$\psi$  in such a way that $P_n$ is a simple pole. Chose now a simple
oriented   broken   line   $P_1,\ldots,P_{n-1}$   on   $\CP$  joining
consecutively  all the singular points of $\psi$ except the last one.
For  every  arc $[P_i,P_{i+1}]$ of this broken line, $i=1,\dots,n-2$,
the  difference  of  their  two preimages defines a relative cycle in
$H_1^-(\hat  S,\{\hat P_1,\dots,\hat P_N\};\Z)$. By construction such
a   cycle   is   antiinvariant  with  respect  to  the  hyperelliptic
involution.  It  is immediate to see that the resulting collection of
cycles   forms   a   basis  in  $H_1^-(\hat  S,\{\hat  P_1,\dots,\hat
P_N\};\Z)$.

Note  that,  a  preimage  of a simple pole does not belong to the set
$\hat   P_1,\ldots,\hat   P_N$.   Thus,   a   preimage   of   an  arc
$[P_i,P_{i+1}]$ having a simple pole as one of the endpoints does not
define  a  cycle  in  $H_1(\hat  S,\{\hat  P_1,\dots,\hat P_N\};\Z)$.
However,  since  a  simple  pole  is  always  a  branching point, the
\textit{difference}  of  the  preimages  of  such  arc  is  already a
well-defined  relative  cycle  in  $H_1(\hat  S,\{\hat P_1,\dots,\hat
P_N\};\Z)$.

Let  $\cQ(d_1,\dots,d_n)$  be the ambient stratum for the meromorphic
quadratic   differential   $(\CP,\psi)$.   The  subspace  $H^1_-(\hat
S,\{\hat     P_1,\dots,\hat     P_N\};\C{})$    in    the    relative
\textit{cohomology}   antiinvariant   with  respect  to  the  natural
involution defines local coordinates in the stratum.

\subsection{Normalization of the volume element.}
\label{sec:subsec:notmalization}

For  any flat surface $S$ in any stratum $\cQ(d_1,\dots,d_k)$ we have
a  canonical  ramified  double  cover  $\hat{S}\to  S$  such that the
induced  quadratic  differential on the Riemann surface $\hat S$ is a
global  square of a holomorphic Abelian differential. We have seen in
\S\ref{ss:coordiantes:in:stratum} that the subspace
$H^1_-(\hat S,\{\hat{P}_1,\ldots,\hat{P}_N\};\C{})$
antiinvariant with respect to the induced action of the hyperelliptic
involution  on  relative cohomology provides local coordinates in the
corresponding     stratum     $\cQ(d_1,\dots,d_n)$    of    quadratic
differentials.     We     define    a    lattice    in    $H^1_-(\hat
S,\{\hat{P}_1,\ldots,\hat{P}_N\};\C{})$ as the subset of those linear
forms  which  take  values  in $\Z\oplus i\Z$ on $H^-_1(\hat S,\{\hat
P_1,\dots,\hat P_N\};\Z)$.

We  define  the  volume element $d\mu$ on $\cQ(d_1,\dots,d_k)$ as the
linear      volume      element      in      the     vector     space
$H^1_-(\hat{M}^2_g,\{\hat{P}_1,\ldots,\hat{P}_N\};\C{})$
normalized  in  such  way  that  the  fundamental domain of the above
lattice has volume $1$.

We warn the reader that for $N>1$ this lattice is a proper sublattice
of index $4^{N-1}$ of the lattice
$$
H^1_-(\widehat S,\{\widehat P_1, \dots, \widehat P_N\};\cx)\ \cap\
H^1(\widehat S,\{\widehat P_1, \dots, \widehat P_N\};\Z\oplus i\Z)\,.
$$
Indeed,  if a flat surface $S$ defines a lattice point for our choice
of  the  lattice,  then the holonomy vector along a saddle connection
joining  distinct  singularities might be half-integer. (However, the
holonomy  vector along any \textit{closed} saddle connection is still
always integer.)

The  choice  of one or another lattice is a matter of convention. Our
choice  makes  formulae  relating enumeration of pillowcase covers to
volumes simpler; see \S~\ref{sec:Jenkins}. Another advantage of
our choice is that the volumes of the strata $\cQ(d,-1^{d+4})$ and of
the  hyperelliptic  components of the corresponding strata of Abelian
differentials  are  the  same  (up to the factors responsible for the
numbering of zeroes and of simple poles).

\begin{Convention}
\label{con:area:1:2}  Similar to the case of Abelian differentials we
choose  a  real  hypersurface  $\cQ_1(d_1,\dots,d_k)$  in the stratum
$\cQ_1(d_1,\dots,d_k)$  of  flat  surfaces  of  fixed  area. We abuse
notation  by  denoting  by  $\cQ_1(d_1,\dots,d_k)$  the space of flat
surfaces  of  area $1/2$ (so that the canonical double cover has area
$1$).
\end{Convention}

The volume element $d\mu$ in the embodying space $\cQ(d_1,\dots,d_k)$
induces  naturally  a  volume  element  $d\mu_1$  on the hypersurface
$\cQ_1(d_1,\dots,d_k)$  in  the  following  way.  There  is a natural
$\cx^\ast$-action        on        $\cQ(d_1,\dots,d_k)$:       having
$\lambda\in\cx^\ast$ we associate to the flat surface $S=(\CP,q)$ the
flat  surface
\begin{equation}
\label{eq:Cstar:action}
\lambda\cdot S:=(\CP,\lambda^2\cdot q)\,.
\end{equation}
In  particular, we can represent any $S\in\cQ(d_1,\dots,d_k)$ as $S =
r  S_{(1)}$, where $r\in\reals_+$, and where $S_{(1)}$ belongs to the
``hyperboloid'': $S_{(1)}\in\cQ_1(d_1,\dots,d_k)$. Geometrically this
means that the metric on $S$ is obtained from the metric on $S_{(1)}$
by  rescaling  with  linear  coefficient  $r$. In particular, vectors
associated  to  saddle connections on $S_{(1)}$ are multiplied by $r$
to  give  vectors  associated  to corresponding saddle connections on
$S$.  It  means  also that $\area(S) = r^2\cdot\area(S_{(1)})=r^2/2$,
since  $\area(S_{(1)})  = 1/2$. We define the \textit{volume element}
$d\mu_1$    on    the   ``hyperboloid''   $\cQ_1(d_1,\dots,d_k)$   by
disintegration of the volume element $d\mu$ on $\cQ(d_1,\dots,d_k)$:
\begin{equation}
\label{eq:disintegration}
d\mu = r^{2n-1} \, dr\, d\mu_1\, ,
\end{equation}
where
$$
2n=\dim_\reals\cQ(d_1,\dots,d_k)=
2\dim_{\C{}}\cQ(d_1,\dots,d_k)=2(k-2)\,.
$$
Using  this  volume element we define the total \textit{volume of the
stratum} $\cQ_1(d_1,\dots,d_k)$:
\begin{equation}
\label{eq:int:cF:nu1}
\Vol\cQ_1(d_1,\dots,d_k):= \int_{\cQ_1(d_1,\dots,d_k)}d\mu_1\,.
\end{equation}

For     a     subset     $E\subset\cQ_1(d_1,\dots,d_k)$     we    let
$C(E)\subset\cQ_1(d_1,\dots,d_k)$ denote the ``cone'' based on $E$:
\begin{equation}
\label{eq:cone}
 C(E):=\{S=rS_{(1)}\,|\, S_{(1)}\in E,\ 0<r\le 1\}\,.
\end{equation}
Our  definition  of  the  volume element on $\cQ_1(d_1,\dots,d_k)$ is
consistent with the following normalization:
\begin{equation}
\label{eq:normalization}
\Vol(\cQ_1(d_1,\dots,d_k)) =
\dim_{\R{}} \cQ(d_1,\dots,d_k)\cdot\mu(C(\cQ_1(d_1,\dots,d_k))\,,
\end{equation}
where  $\mu(C(\cQ_1(d_1,\dots,d_k))$  is  the  total  volume  of  the
``cone''  $C(\cQ_1(d_1,\dots,d_k))\subset\cQ(d_1,\dots,d_k)$ measured
by means of the volume element $d\mu$ on $\cQ(d_1,\dots,d_k)$ defined
above.

\subsection{Reduction    of    volume    calculation    to   counting
lattice points}
\label{subsec:integer:points}

The   volume   of   a   stratum   $\cQ_1(d_1,\dots,d_k)$  is  defined
by~\eqref{eq:normalization} as
$$
\Vol\cQ_1(d_1,\dots,d_k) =
\dim_{\R{}} \cQ(d_1,\dots,d_k)\cdot\mu(C(\cQ_1(d_1,\dots,d_k))\,,
$$
where  $\mu(C(\cQ_1(d_1,\dots,d_k))$  is  the  total  volume  of  the
``cone''  $C(\cQ_1(d_1,\dots,d_k))\subset\cQ(d_1,\dots,d_k)$ measured
by means of the volume element $d\mu$ on $\cQ(d_1,\dots,d_k)$ defined
in  \S\ref{sec:subsec:notmalization}.  The  total volume of the
cone  $C(\cQ_1(d_1,\dots,d_k))$  is  the  limit  of the appropriately
normalized Riemann sums.

The  volume  element  $d\mu$ is defined as a linear volume element in
cohomological  coordinates,  normalized  by certain specific lattice.
Chose  a  positive  $\epsilon$ such that $1/\epsilon$ is integer, and
consider   a   sublattice   of   the   initial   lattice   of   index
$(1/\epsilon)^{\dim_{\R{}}  \cQ(d_1,\dots,d_k)}$  partitioning  every
side   of   the   initial   lattice  into  $1/\epsilon$  pieces.  The
corresponding  Riemann  sums  count  the  number  of  points  of  the
sublattices  which  get  inside  the cone. Thus, by definition of the
measure $\mu$ we get
\begin{multline*}
\mu(C(\cQ_1(d_1,\dots,d_k))=
\lim_{\epsilon\to 0}
\epsilon^{\dim_{\R{}} \cQ(d_1,\dots,d_k)}\cdot
\\
\big(\text{Number of points of the $\epsilon$-sublattice inside the cone }
C(\cQ_1(d_1,\dots,d_k))\big)\,.
\end{multline*}

We  assume that $1/\epsilon$ is integer. Note that a flat surface $S$
represents  a  point  of  the  $\epsilon$-lattice, if and only if the
surface     $(1/\epsilon)\cdot     S$     (in     the     sense    of
definition~\eqref{eq:Cstar:action}) represents a point of the integer
lattice.  Denoting  by  $C(\cQ_\Deg(d_1,\dots,d_k))$  the set of flat
surfaces   in  the  stratum  $\cQ(d_1,\dots,d_k)$  of  area  at  most
$\Deg/2$, and taking into consideration that
$$
\area((1/\epsilon)\cdot S)=1/\epsilon^2\cdot\area(S)
$$
we can rewrite the above relation as
\begin{multline}
\label{eq:lattice:points}
\mu(C(\cQ_1(d_1,\dots,d_k))=
\lim_{\Deg\to+\infty}
\Deg^{-\dim_{\C{}} \cQ(d_1,\dots,d_k)}\cdot
\\
\big(\text{Number of lattice points inside the cone }
C(\cQ_\Deg(d_1,\dots,d_k)\big)\,.
\end{multline}

\section{Counting generalized Jenkins--Strebel differentials}
\label{sec:Jenkins}

In  this  section  we pass to counting the pillowcase covers. We have
seen  in  \S\ref{subsec:integer:points}  that  volume  calculation is
equivalent     to     counting     the     lattice     points.     In
\S\ref{subsec:lattice:square:tiled:pillowcase}  we  discuss  in  more
details  the  \textit{pillowcase  covers}  and  show that counting of
lattice  points  is equivalent to the counting problem for pillowcase
covers.  Starting  from  section  \S\ref{sec:subsec:jenkins}  we work
exclusively   with   the   strata  $\cQ(1^K,-1^{K+4})$.

In~\S\ref{sec:subsec:jenkins}   we   discuss   in  more  details  the
functions   $F_{n,m}$  from  Theorem~\ref{theorem:local:poly},  study
their elementary properties and prove formula~\eqref{eq:eqn:k:layer}.
We  consider in \S\ref{subsubsec:example} a particular case $F_{2,2}$
corresponding  to Figure~\ref{fig:graph1} as an example, for which we
perform       an       explicit       (by      hand)   computation.
In~\S\ref{subsubsec:kontsevich}  we  obtain  a general expression for
$F_{m,0}$  as  a  corollary from a theorem of Kontsevich~\cite{Kontsevich}. We use this
expression      as     a     base     of     recursion    developed
in~\S\ref{subsec:recurrense}, where we express $F_{m+1,n+1}$ in terms
of    $F_{m,n}$.    This    recurrence relation    allows    us    to    prove
in~\S\ref{subsec:recurrense}        Theorem~\ref{theorem:local:poly}.
Finally,  in~\S\ref{sec:subsec:total}  we  compute  the  sum over all
decorated   trees   and   prove   the   main   identity   stated   in
Theorem~\ref{th:volume:through:pillows}.  In  \S\ref{subsec:1:5}  and
\S\ref{subsec:2:6} we illustrate our formula for concrete examples of
the strata $\cQ(1,-1^5)$ and $\cQ(1^2,-1^6)$ correspondingly.

\subsection{Lattice  points,  square-tiled  surfaces,  and pillowcase
covers}
\label{subsec:lattice:square:tiled:pillowcase}

Let  $\Lambda \subset \cx$ be a lattice, and let $\T^2 = \cx/\Lambda$
be the associated torus. The quotient
$$
\cP : = \T^2/\pm
$$
by  the  map  $z  \rightarrow  -z$ is known as the \textit{pillowcase
orbifold}.  It  is  a  sphere with four $(\Z/2)$-orbifold points (the
corners  of  the  pillowcase). The quadratic differential $(dz)^2$ on
$\T^2$  descends  to  a  quadratic differential on $\cP$. Viewed as a
quadratic  differential  on  the  Riemann sphere, $(dz)^2$ has simple
poles  at  corner  points. When the lattice $\Lambda$ is the standard
integer  lattice $\Z\oplus i\Z$, the flat torus $\T^2$ is obtained by
isometrically  identifying  the  opposite sides of a unit square, and
the  pillowcase  $\cP$  is  obtained by isometrically identifying two
squares    with    the    side    $1/2$    by   the   boundary,   see
Figure~\ref{fig:square:pillow}.

Consider  a  connected  ramified cover $\hat\cP$ over $\cP$ of degree
$\Deg$  having  ramification  points  only  over  the  corners of the
pillowcase.  Clearly,  $\hat\cP$  is  tiled by $2\Deg$ squares of the
size $(1/2)\times(1/2)$ in such way that the squares do not superpose
and  the vertices are glued to the vertices. Coloring the two squares
of the pillowcase $\cP$ one in black and the other in white, we get a
chessboard  coloring of the square tiling of the the cover $\hat\cP$:
the white squares are always glued to the black ones and vice versa.

\begin{lemma}
\label{lm:pillowcase}
Let $S$ be a flat surface in the stratum $\cQ(d_1,\dots,d_k)$.
The following properties are equivalent:
\begin{enumerate}
\item
The surface $S$ represents a lattice point in $\cQ(d_1,\dots,d_k)$;
\item
$S$ is a cover over $\cP$ ramified only over the corners of the pillow;
\item
$S$ is tiled by black and white $(1/2)\times(1/2)$ squares respecting
the chessboard coloring.
\end{enumerate}
\end{lemma}
\begin{proof}
We  have  just proved that (2) implies (3). To prove that (1) implies
(2)  we  define  the following map from $S$ to $\cP$. Fix a zero or a
pole  $P_0$  on  $S$.  For  any  $P\in S$ consider a path $\gamma(P)$
joining  $P_0$  to  $P$  having  no  self-intersections and having no
zeroes or poles inside. The restriction of the quadratic differential
$q$   to   such   $\gamma(P)$   admits  a  well-defined  square  root
$\omega=\pm\sqrt{q}$,  which is a holomorphic form on the interior of
$\gamma$. Define
$$
P\mapsto \left(\int_{\gamma(P)} \omega\mod \Z\oplus i\Z\right)/\pm\,.
$$
Of  course,  the  path  $\gamma(P)$ is not uniquely defined. However,
since  the  flat  surface  $S$  represents  a  lattice point (see the
definition in \S\ref{sec:subsec:notmalization}), the difference
of  the  integrals  of $\omega$ over any two such paths $\gamma_1(P)$
and  $\gamma_2(P)$  belongs to $\Z\oplus i\Z$, so taking the quotient
over the integer lattice and over $\pm$ we get a well-defined map. By
definition  of  the pillowcase $\cP$ we have, $\cP=\left(\C{}\!\!\mod
\Z\oplus  i\Z\right)/\pm$.  Thus, we have defined a map $S\to\cP$. It
follows  from the definition of the map, that it is a ramified cover,
and  that  all  regular  points  of  the flat surface $S$ are regular
points  of  the cover. Thus, all ramification points are located over
the corners of the pillowcase.

A similar consideration shows that (3) implies (1).
\end{proof}

Let $\operatorname{Sq}_\Deg(d_1,\dots,d_k)$ be the number of surfaces
in  the  stratum $\cQ(d_1,\dots,d_k)$ tiled with at most $\Deg$ black
and   $\Deg$   white  squares  respecting  the  chessboard  coloring.
Lemma~\ref{lm:pillowcase}          allows          to         rewrite
formula~\eqref{eq:lattice:points} as follows:
$$
\mu(C(\cQ_1(d_1,\dots,d_k))=
\lim_{\Deg\to+\infty}
\Deg^{-\dim_{\C{}} \cQ(d_1,\dots,d_k)}\cdot
\operatorname{Sq}_\Deg(d_1,\dots,d_k)
\,.
$$
Taking into consideration~\eqref{eq:normalization} we get
\begin{multline}
\label{eq:Vol:through:square:tiled}
\Vol\cQ_1(d_1,\dots,d_k)=2\dim_\cx\cQ(d_1,\dots,d_k)\,\cdot
\\
\lim_{\Deg\to+\infty}
\Deg^{-\dim_{\C{}} \cQ(d_1,\dots,d_k)}\cdot
\operatorname{Sq}_\Deg(d_1,\dots,d_k)
\,.
\end{multline}

\subsection{Local Polynomials}
\label{sec:subsec:jenkins}

In  order  to  count pillowcase covers we note that if $\hat\cP$ is a
square-tiled  pillowcase  cover,  it can be decomposed into cylinders
with  integer widths, with zeros and poles lying on the boundaries of
these  cylinders.  We  call  these boundaries singular layers. We can
form an associated graph whose vertices are singular layers and edges
are  cylinders.  For  a  pillowcase  cover $\hat\cP$ in $\cQ( 1^\noz,
(-1)^{\noz+4})$  the associated graph will be a tree, since $\hat\cP$
is   a   sphere.  Figure~\ref{fig:global:graph:and:layers}  gives  an
example of such a tree.

Suppose  there  are $k$ cylinders of width $w_1,\dots w_k$ and height
$h_1,\dots,h_k$  respectively. Since $\hat\cP$ is a sphere, there are
$k+1$  singular layers in the decomposition of $\hat\cP$. Fix the way
in  which  our \textit{labelled} (\textit{named}) zeroes and poles are
distributed  through  singular  layers  (vertices  of the global tree
$\Tree$).

\begin{lemma}
\label{lm:number:of:pillowcase:covers:of:given:type}
The  total  number of pillowcase covers of degree at most $\Deg$ with
a decomposition of a fixed type can be written as
\begin{equation}
\label{eq:k:layer}
2^k
\sum_{\substack{w\cdot h \le\Deg\\
w_i,h_i\in\N}}
w_1\cdot w_2\dots \cdot w_k\cdot
\prod_{i=1}^{k+1} F_i(w_1, \ldots, w_k)\,,
\end{equation}
where  $F_j$  is a function counting the number of ways the cylinders
of width $w_i$ can be glued at vertex $j$
\end{lemma}
Here $F_j$ depends only on the widths $w_{i_j}$ associated
to edges adjacent to vertex $j$.
\begin{proof}
Every  cylinder  is  determined  by  the  following parameters: by an
integer   perimeter   (length   of  the  waist  curve)  $w_i$;  by  a
half-integer  hight  $h_i$  and  by a half-integer twist $t_i$, where
$0\le  t_i\le  w_i$.  Thus, there are $2w_i$ choices for the value of
the  twist  $t_i$,  which  explains  the factor $(2w_1) (2w_2) \ldots
(2w_k) = 2^k \prod_{i=1}^k w_i$.

The  restriction  on the area $\sum h_i\cdot w_i\le N/2$ with integer
$w_i$   and   \textit{half-integer}   $h_i$   is  equivalent  to  the
restriction   $\sum   h_i\cdot  w_i\le  N$  with  integer  $w_i$  and
\textit{integer} $h_i$.
\end{proof}

Our  current  goal  is  to  show  that up to terms of lower order the
counting  function  $F_{m,n}$  associated  to a layer with $m$ simple
zeros   and   $n$  first  order  poles,  is  the  explicit  symmetric
polynomial~\eqref{eq:local:poly}.  We  emphasize  that  the zeros and
poles are \textit{labelled}.


The  neighborhood of a singular layer with $m$ zeros and $n$ poles is
a  metric  ribbon  graph  with  $m$  trivalent vertices (representing
zeroes),  $n$ univalent vertices (representing simple poles) and with
$l$               boundary               components,              see
Figure~\ref{fig:global:graph:and:layers}.  The  width  $w_i$  of each
boundary  component  is  given by the sum of the lengths of the edges
lying  on the boundary. Thus, given a collection of integer widths of
cylinders, the counting problem can be restated as finding the number
of  graphs with half-integer edge lengths yielding these widths. This
is  a  system  of  linear  equations, and the number of half-integral
solutions  is  equal to the volume of the space of all real solutions
for the edge lengths.

Note that the neighborhood of a singular layer with $m$ zeros and $n$
poles can be also viewed as is a topological sphere with $n+m$ marked
points  and with $l$ punctures. This sphere is endowed with a complex
structure;    the   corresponding   $\CP$   carries   a   meromorphic
Jenkins--Strebel  differential  having  $m$  simple zeros, $n$ simple
poles,  and  $l$  double  poles  (which  are  \textit{not}  poles  of
$\hat\cP$)  corresponding  to  $l$  cylinders  of widths $w_i, i = 1,
\ldots  l$.  The number of cylinders $l$ is specified by the relation
$m - n -2l = -4$, that is,
$$
l = \frac{(m-n)}{2} +2\,.
$$
By~\cite{Strebel},  there  is a bijective correspondence between such
Jenkins--Strebel differentials and metric ribbon graphs on the sphere
with  $m$  trivalent  and  $n$  univalent  vertices.  To  count these
differentials, we follow an approach of Kontsevich~\cite{Kontsevich}.

Given  a  ribbon  graph  on  the  sphere  with  $m$ trivalent and $n$
univalent  vertices, we have $v = m+n$, $e = (3m+n)/2$, where $e$ and
$v$  are  the  number of edges and vertices respectively. Letting $f$
denote  the  number  of  faces  (i.e, complementary regions), we have
$v-e+f=2$,  so $(n-m)/2 + f =2$, i.e., $2f = 4+m-n$. This imposes the
restriction  that  $m-n  \in  2\Z$ and that $m-n > -4$. Also, we have
$e-f=  v-2$,  which  suggests our polynomial should be a degree $v-2$
polynomial in $f$ variables.


\addtocounter{subsection}{1}
\label{subsubsec:example}
\addcontentsline{toc}{subsection}{\thesubsection.\quad Example: direct computation of $F_{2,2}$}

\bigskip\noindent
\thesubsection.~\textbf{Example: direct computation of $\pmb{F_{2,2}}$.}
Let  us  explicitly  compute the~local polynomial $F_{2,2}(w_1,w_2)$.
The  list  of  connected ribbon graphs having two vertices of valence
$3$  and  two  vertices of valence $1$ with \textit{labelled boundary
components}   is   presented   at   Figure~\ref{fig:graph1}.   Note  that
interchanging the labelling of the boundary components for the ribbon
graphs  $\Gamma_{1,1}$  and  $\Gamma_{2,1}$  we  get different ribbon
graphs   $\Gamma_{1,2}$  and  $\Gamma_{2,2}$  correspondingly,  while
changing the labelling of the boundary components of the ribbon graph
$\Gamma_3$ we get an isomorphic ribbon graph.

Note,  that  since  our  ribbon graphs represent singular layers on a
topological  sphere,  they  are  always  planar,  i.e.,  they  can be
embedded into a plane.

Consider,  for  example, the graph $\Gamma_{1,1}$ on top on the left.
The  widths  of  the cylinders are given by $w_1 = l_1 + l_2$, and by
$w_2  =  l_1 + l_2 + 2l_3 + 2l_4$, so $\Gamma_1$ is realizable if and
only  if  $w_1<w_2$.  Given  $w_1<w_2$  there are $2w_1$ half-integer
positive  solutions $l_1, l_2$ of equation $w_1 = l_1 + l_2$, and for
each  such solution there are $w_2-w_1$ half-integer solutions of the
equation  $w_2  =  l_1  +  l_2  +  2l_3  +  2l_4$.  Thus,  the impact
$F_{\Gamma_{1,1}}(w_1,w_2)$ of $\Gamma_{1,1}$ to the local polynomial
$F_{2,2}(w_1,w_2)$ has the form
$$
F_{\Gamma_{1,1}}(w_1,w_2):=\begin{cases}
0&\text{when }w_1\ge w_2\\
2w_1(w_2-w_1)&\text{when }w_1<w_2
\end{cases}\,.
$$

Note  that  the  number of quadruples of positive half-integers $l_1,
l_2,  l_3,  l_4$ satisfying the above equations, can be viewed as the
volume  of  the  associated  region of solutions in the positive cone
$\R_{>0}^{4}$.       Consider       the       Laplace       transform
$\widehat{F}_{\Gamma_{1,1}}(\lambda_1,   \lambda_2)   =   \int_{\R^2}
e^{-\lambda\cdot    w}    F_{\Gamma_{1,1}}(w_1,   w_2)   dw$.   Since
$F_{\Gamma_{1,1}}(w_1,  w_2) = 0$ for $w_1, w_2 <0$ and for $w_1 \geq
w_2$,  and since $w_1 = l_1 + l_2$, $w_2 = l_1 +l_2 + 2l_3 +2l_4$, we
obtain
\begin{multline}
\label{eq:example}
\widehat{F}_{\Gamma_{1,1}}(\lambda_1, \lambda_2)
=
2^3\cdot\int_{l_1, l_2, l_3, l_4 >0}
e^{-\lambda_1(l_1 +l_2)}e^{-\lambda_2(l_1+l_2 + 2l_3 +2l_4)}
dl_1dl_2dl_3dl_4 \\
=
\frac{1}{2\cdot}
\frac{2}{\lambda_1 +\lambda_2}\cdot
\frac{2}{\lambda_1+\lambda_2}\cdot
\frac{2}{2\lambda_2}\cdot
\frac{2}{2\lambda_2}\,,
\end{multline}
where the factor
$$
2^{m+n-1}=2^3
$$
in  front  of the integral comes from the normalization of the volume
element  in  cohomological  coordinates. This coefficient can also be
seen,  in  general, as follows: we have $e= \frac{3m+n}{2}$ edges and
$f=  \left(\frac{m-n}{2}  + 2\right)$ faces (adjacent cylinders). The
latter  give  relations  between  edge lengths; the difference is our
dimension  $$\frac{3m+n}{2}  -  \left(\frac{m-n}{2}+2\right)=m+n-2.$$
However,  in  the  parity count the relations are not independent. If
all  edge  lengths  are half-integer, and all perimeters of cylinders
but  one  are  integer, the last perimeter is \emph{automatically} an
integer.  To  see  this  compute the sum of the lengths of perimeters
with  natural  signs. If all edge lengths are half-integer, all edges
which  separate different cylinders get cancelled in this sum and all
other  edges  are  counted  with  a factor of 2. Thus, the sum of the
residues  is integer. This implies that if all perimeters but one are
integer,  the  last  one  is automatically an integer, and we go from
$(m+n-2)$ to $(m+n-1)$.

The      expressions      for      $F_{\Gamma_{1,2}}$     and     for
$\hat{F}_{\Gamma_{1,2}}$     are     symmetric     to    those    for
$F_{\Gamma_{1,1}}$ and $\hat{F}_{\Gamma_{1,1}}$ respectively. Similar
calculations provide the following answers for the remaining graphs:
$$
F_{\Gamma_{2,1}}(w_1,w_2):=\begin{cases}
0&\text{when }w_1\ge w_2\\
\frac{(w_2-w_1)^2}{2}&\text{when }w_1<w_2
\end{cases}\qquad
\widehat{F}_{2,1}(\lambda_1,\lambda_2)=
\frac{1}{2}\frac{2}{\lambda_1+\lambda_2}\left(\frac{1}{\lambda_2}\right)^3
\,.
$$
The      expressions      for      $F_{\Gamma_{2,2}}$     and     for
$\hat{F}_{\Gamma_{2,2}}$     are     symmetric     to    those    for
$F_{\Gamma_{2,1}}$    and    $\hat{F}_{\Gamma_{2,1}}$   respectively.
Finally,
$$
F_{\Gamma_3}(w_1,w_2):=\begin{cases}
w_2^2&\text{when }w_1\ge w_2\\
w_1^2&\text{when }w_1<w_2
\end{cases}
\qquad
\widehat{F}_{\Gamma_3}(\lambda_1,\lambda_2)=
\frac{1}{2}\left(\frac{2}{\lambda_1+\lambda_2}\right)^2
\frac{1}{\lambda_1}\frac{1}{\lambda_2}
\,.
$$

There  are $2!\cdot 2!$ ways to give names to $2$ zeroes (i.e. to $2$
trivalent vertices) and to $2$ poles (i.e. to $2$ univalent vertices)
of  the graphs $\Gamma_{2,1}, \Gamma_{2,2}$ and $\Gamma_3$, and there
is  $\frac{1}{2}\cdot  2!\cdot  2!$  ways to give names to zeroes and
poles   of   the   graphs  $\Gamma_{1,1},  \Gamma_{1,2}$.  Thus,  the
contribution of all graphs to $F_{2,2}$ is
\begin{multline*}
F_{2,2}(w_1,w_2)=2!\cdot 2!\left(
\frac{1}{2}\cdot F_{\Gamma_{1,1}}+\frac{1}{2}\cdot F_{\Gamma_{1,2}}+
F_{\Gamma_{2,1}}+F_{\Gamma_{2,2}}+
F_{\Gamma_3}\right)
=\\
4\left(
F_{\Gamma_{1,1}}+
F_{\Gamma_{2,1}}+
F_{\Gamma_3}\right)\quad\text{when }w_1<w_2
=\\
4
\left(w_1(w_2-w_1)+\frac{(w_2-w_1)^2}{2}+w_1^2\right)=
2(w_1^2+w_2^2)\quad\text{when }w_1<w_2\,.
\end{multline*}
Similarly,
\begin{multline*}
F_{2,2}(w_1,w_2)=4\left(
F_{\Gamma_{1,2}}+
F_{\Gamma_{2,2}}+
F_{\Gamma_3}\right)\quad\text{when }w_1>w_2
=\\
4\left(w_2(w_1-w_2)+\frac{(w_1-w_2)^2}{2}+w_2^2\right)=
2(w_1^2+w_2^2)\quad\text{when }w_1>w_2\,.
\end{multline*}
We observe that, though for individual graphs $\Gamma$ the expression
$F_\Gamma(w_1,w_2)$  is  not  symmetric  in  $w_1,w_2$, the total sum
$F_{2,2}(w_1,w_2)$ is a symmetric polynomial in $w_1,w_2$.

Note that formally speaking, we have calculated only the leading term
of  the  local  polynomial neglecting a small correction arising from
degenerate  solutions  when  one or several $l_i$ vanish. A.~Okounkov
and  R.~Pandharipande prove in~\cite{OP} that counting the degenerate
solutions  in  an  appropriate way we get a true symmetric polynomial
$F_{m,n}$  not  only  in the leading term, but exactly. Since for the
purposes of counting the volume we are interested only in the leading
term of the asymptotics, we neglect this subtlety.

We   could   also   compute  $\widehat{F}_{2,2}(\lambda_1,\lambda_2)$
directly. The advantage of this calculation is that we do not need to
follow  the  system of inequalities, which becomes quite involved for
complicated graphs. In our case we get

\begin{multline*}
\widehat{F}_{2,2}(\lambda_1,\lambda_2)=2!\cdot 2!\left(
\frac{1}{2}\cdot \widehat{F}_{\Gamma_{1,1}}+\frac{1}{2}\cdot \widehat{F}_{\Gamma_{1,2}}+
\widehat{F}_{\Gamma_{2,1}}+\widehat{F}_{\Gamma_{2,2}}+
\widehat{F}_{\Gamma_3}\right)
=\\
\left(\frac{2}{\lambda_1+\lambda_2}\right)^2\left(\frac{1}{\lambda_1}\right)^2+
\left(\frac{2}{\lambda_1+\lambda_2}\right)^2\left(\frac{1}{\lambda_2}\right)^2+
2\cdot\frac{2}{\lambda_1+\lambda_2}\left(\frac{1}{\lambda_1}\right)^3+
\\
2\cdot\frac{2}{\lambda_1+\lambda_2}\left(\frac{1}{\lambda_2}\right)^3+
2\cdot\left(\frac{2}{\lambda_1+\lambda_2}\right)^2\frac{1}{\lambda_1}\frac{1}{\lambda_2}
=4\left(
\frac{1}{\lambda_1\lambda_2^3}+\frac{1}{\lambda_2\lambda_1^3}
\right)\,.
\end{multline*}

\subsection{Kontsevich's Theorem}
\label{subsubsec:kontsevich}
Consider  now the general setting. As above, let $\Gamma$ be a ribbon
graph  on the sphere, and let $w_1, \ldots, w_l$ be the widths of the
complementary  regions.  Let  $F_{\Gamma}(w_1,  \ldots,  w_l)$ be the
volume  of  the region in $\R_{>0}^{|e|}$ corresponding to lengths of
edges  so  that  the  sum  of the edges adjacent to the region $i$ is
$w_i$.  Here,  $|e|$  is  the number of edges of $\Gamma$. Taking the
Laplace transform, we define
\begin{equation}
\label{eq:hatf:gamma:def}
\widehat{F}_{\Gamma}(\lambda) =
2^{m+n-1}
\cdot
\prod_{e \in \Gamma}
\frac{1}{\tilde{\lambda}(e)}\,,
\end{equation}
where  the  product  is  taken  over  all  the edges $e$ of the graph
$\Gamma$, and $\tilde{\lambda}(e)$ denotes the sum of the $\lambda$'s
corresponding to width variables associated to regions bordering $e$. Our normalization is as in (\ref{eq:example}).
If  $e$  is  an  edge adjacent to a univalent vertex, then it is only
bordered  by  one  region. Let $F_{m,n}(w_1, \ldots, w_l)$ denote the
total  volume  (that  is,  the  sum  over  all possible ribbon graphs
$\Gamma$  with  $m$  trivalent  and  $n$ univalent vertices), and let
$\widehat{F}_{m,n}$ denote the Laplace transform of $F_{m,n}$.
\begin{equation}
\label{eq:hatf:def}
\widehat{F}_{m,n}(\lambda) =
\sum_{\Gamma}\widehat{F}_{\Gamma} (\lambda) =
2^{m+n-1}
\cdot
\sum_{\Gamma}\prod_{e \in \Gamma}
\frac{1}{\tilde{\lambda}(e)}\,,
\end{equation}
where  the  sums  are taken over all connected ribbon graphs $\Gamma$
having  $m$  trivalent  vertices, $n$ univalent vertices and no other
vertices. We have:
   \medskip
\begin{theorem}~\cite[\S 3.1, page 10]{Kontsevich}
\label{theorem:kontsevich}
Let $m=2k, l=k+2$. Then
$$
\widehat{F}_{m, 0}(\lambda_1, \ldots, \lambda_l) =
2^{k-1} m!  \sum_{\substack{k_1,\dots,k_l\\k_1+\dots+k_l = k-1}}
\binom{k-1}{k_1, \ldots, k_l} \prod_{i=1}^{l}
\frac{(2k_i-1)!!}{\lambda_i^{2k_i +1}}\,,
$$
where   $(2k_i  -1)!!  =  (2k_i  -1)(2k_i  -3)(2k_i  -5)  \ldots  1$.
\end{theorem}
\noindent
Recall that by convention $(-1)!!=0!=1$.
\smallskip

\noindent  \textbf{Remark:} Kontsevich states this result in terms of
certain                      intersection                     numbers
$\langle\tau_{k_1}\ldots\tau_{k_l}\rangle$  in  place of the binomial
coefficients  $\binom{k-1}{k_1,  \ldots, k_l}$. However, the equality
of  the  two quantities was known to Witten~\cite{Witten}, see, also e.g,~\cite[equation 2.3]{OP5}. There are
also  two  differences  in  normalization- first, we are working with
\emph{labelled}  zeros  (and  later  also  poles),  and  edges, which
eliminates  any  symmetry  group factors and adds the factor of $m!$,
and  we  normalize the half-integer lattice to have volume $1$, which
accounts for the difference in the factor of a power of $2$.

\begin{corollary}
\label{cor:m:equals:zero}
Theorem~\ref{theorem:local:poly}   and  formula~\eqref{eq:local:poly}
are valid for $n=0$:
\begin{equation}
\label{eq:local:poly:n:equals:0}
F_{m,0}(w_1, w_2, \ldots
w_l) = m! \sum_{\sum_{i=1}^{l} k_i = k-1} \binom{k-1}{k_1,
\ldots, k_l} \prod_{i=1}^{l} \frac{w_i^{2k_i}}{k_i!}\,,
\end{equation}
\end{corollary}
\begin{proof}
Taking  Laplace  transforms,  and  noting  that if $F(x) = x^k$, then
$\widehat{F}(x) = k!/\lambda^{k+1}$, we obtain
\begin{equation}
\label{nopoles}
F_{m,0} = m!2^{k-1}\sum_{\sum_{i=1}^{l} k_i = k-1}
\binom{k-1}{k_1, \ldots, k_l} \prod_{i=1}^{l}
\frac{(2k_i-1)!!}{(2k_i)!}w_i^{2k_i}\,.
\end{equation}
The product inside the summation can be simplified using
\begin{eqnarray}
\label{comb}
\frac{(2k_i-1)!!}{(2k_i)!} &=& \frac{(2k_i -1)(2k_i -3)
\ldots 1}{(2k_i)(2k_i -1)(2k_i-2)(2k_i -3) \ldots 1} \nonumber\\ &=&
\frac{1}{(2k_i)(2k_i-2)(2k_i -4)\ldots 2} =
\frac{1}{2^{k_i} {k_i!}}\,.
\end{eqnarray}
Noting  that  $\sum  k_i  =  k-1$  allows  us to cancel the $2^{k-1}$
factor,  yielding~\eqref{eq:local:poly:n:equals:0}  which corresponds
to the case $n=0$ of our main theorem.
\end{proof}

\subsection{Recurrence relations and evaluation of local polynomials}
\label{subsec:recurrense}
To  prove  our  general  formula~\eqref{eq:local:poly}  for arbitrary
local  polynomials $F_{m,n}$, we require another lemma which gives an
induction relating $F_{m+1, n+1}$ to $F_{m, n}$.

\begin{lemma}
\label{induction}
Fix  notation  as  in  Theorem~\ref{theorem:local:poly}. Then for any
nonnegative   $m,n\in\Z$,  not  simultaneously  equal  to  zero,  the
function $F_{m,n}$ is a polynomial in $w_i$ satisfying the relation
$$
F_{m+1,n+1} = 2(m+1)\cdot D (F_{m,n})\,,
$$
where  $  D=\sum_{i=1}^k  D_{w_i}$;  operators  $D_w$  are defined on
monomials  by  $D_w(w^n)=\frac{w^{n+2}}{n+2}$  and  are  extended  to
arbitrary polynomials by linearity.
\end{lemma}

\begin{remark}
Note  that  the  number  of variables $l$ does not change, as it only
depends on the difference $m-n$.
\end{remark}

\begin{proof}
In  terms  of  Laplace transforms, the statement of the lemma becomes
\begin{equation}
\label{polestring}
\widehat{F}_{m+1, n+1} =
2(m+1)\cdot
\sum_{i=1}^{l}
\frac{-1}{\lambda_i} \frac{\partial}{\partial \lambda_i}
\widehat{F}_{m,n}\,.
\end{equation}
To  prove  this,  we  proceed  at a graph-by-graph level. Fix a graph
$\Gamma$  with  $m$  trivalent  and  $n$  univalent labelled vertices.
Define $p_{ij}(\Gamma)$ as the number of edges of $\Gamma$ separating
regions  $i$  and $j$. By formula~\eqref{eq:hatf:gamma:def} (see also
the Example in \S\ref{subsubsec:example})
\begin{equation}
\label{eq:f:hat:gamma}
\widehat{F}_{\Gamma} =
2^{m+n-1}
\cdot\prod_{i\le j}
\left(\frac{
1
}{\lambda_i +\lambda_j}\right)^{p_{i,j}}\,.
\end{equation}

Let  $\Gamma_{i,j}$  be  the  graph  with  $m+1$  trivalent and $n+1$
univalent  vertices  formed  by  adding  a new edge in the region $j$
(corresponding to $\lambda_j$), so that the new trivalent vertex lies
on  an  edge  adjacent  to  the regions $i$ and $j$ (corresponding to
$\lambda_i$     and     $\lambda_j$;    possibly    $i=    j$).    By
formula~\eqref{eq:hatf:gamma:def}    (see   also   the   Example   in
\S\ref{subsubsec:example})
$$
\widehat{F}_{\Gamma_{i,j}} =
2^2\cdot\frac{1}{2\lambda_j}\cdot
\frac{1}{\lambda_i +\lambda_j}\cdot
\widehat{F}_{\Gamma}=
\frac{1}{\lambda_j}\cdot
\frac{2}{\lambda_i +\lambda_j}\cdot
\widehat{F}_{\Gamma}\,.
$$

We  may assume that the ``new'' pole (univalent vertex) is located at
the  end  of  the  new  edge. However, there are $m+1$ choices of the
simple  zero  (trivalent  vertex)  at  the other extremity of the new
edge. From now on we will fix the labeling of the vertices of the new
graph, and we multiply the final result by this factor $(m+1)$.

Summing  the above formula over all edges adjacent to the region $j$,
we obtain the contribution $\widehat{F}_{\Gamma}^{(j)}$ associated to
attaching  a new univalent vertex in region $j$. Note, that the edges
having region $j$ on both sides should be counted twice, since we can
attach the new edge on both sides of the original edge, producing two
different graphs. Thus,
$$
\widehat{F}_{\Gamma}^{(j)}=
2p_{jj}(\Gamma)\cdot\frac{1}{\lambda_j^2}\cdot\widehat{F}_\Gamma+
\sum_{i\neq j} p_{ij}(\Gamma)\cdot
\frac{1}{\lambda_j}\cdot\frac{2}{\lambda_i+\lambda_j}\cdot\widehat{F}_\Gamma\,.
$$
Applying                         the                         operator
$\frac{-2}{\lambda_j}\cdot\frac{\partial}{\partial        \lambda_j}$
to~\eqref{eq:f:hat:gamma}  we  obtain  exactly  the  same expression.
Taking  into  consideration  the  factor  $(m+1)$ responsible for the
numbering we prove relation~\eqref{polestring}. Inverting the Laplace
transform and applying Corollary~\ref{cor:m:equals:zero} and explicit
evaluation  $F_{0,2}=F_{1,1}=1$  as  the  base  of the recurrence, we
complete the proof of the statement of the Lemma.
\end{proof}

\begin{proof}[Proof of Theorem~\ref{theorem:local:poly}]
We    first    consider    the    case    $m>n$.    We    know,    by
Corollary~\ref{cor:m:equals:zero},  that
\begin{equation}
\label{base}
F_{m-n,  0} = (m-n)! \sum_{\sum_{i=1}^{l} k_i = k-1} \binom{k-1}{k_1,
\ldots,  k_l}  \prod_{i=1}^{l} \frac{w_i^{2k_i}}{k_i!}\,,
\end{equation}
see~\eqref{eq:local:poly:n:equals:0}, where $k$ and $l$ are as in the
statement  of Theorem~\ref{theorem:kontsevich}. Our result follows by
applying   Lemma~\ref{induction}   $n$  times  to  (\ref{base}),  and
by observing that the operator
$$
\prod_{\sum n_i =n} D_i^{n_i}
$$
transforms      the      term      $\frac{w_i^{2k_i}}{k_i!}$     into
$$
\frac{w_i^{2(k_i+n_i)}}{(2k_i +2n_i)(2k_i
+ 2n_i -2) \ldots (2k_i +2)k_i!} =
\frac{1}{2^{n_i}}\frac{w_i^{2(k_i+n_i)}}{(k_i+n_i)!}\,.
$$
Combining  the  factors  of  $2$, we obtain a $\frac{1}{2^n}$. On the
outside,  we  obtain  the factors $(2m) (2(m-1)) \ldots (2(m-n +1))$,
which,  combined  with  the  $(m-n)!$, yields \mbox{$2^n (m-n)!$}, so
cancelling the $2^n$ factors, we obtain
\begin{equation}
\label{eq:term}
F_{m,n} = m!
\displaystyle\sum_{\substack{\sum_{i=1}^l k_i = k-1\\ \sum_{i=1}^l n_l = n}}
\binom{n}{ n_1, \ldots, n_l}\binom{k-1}{k_1, \ldots, k_l}
\prod_{i=1}^{l} \frac{(w_i)^{2(k_i+n_i)}}{(k_i +n_i)!}\,.
\end{equation}
Rewriting~\eqref{eq:term}  by  multiplying and dividing by the factor
$a_1!a_2!\ldots a_l!$\,, where $a_i=k_i+n_i$, and using the resulting
factors  $a_i!$  to rearrange  multinomial coefficients as products of
binomial coefficients we get
\begin{equation}\label{eq:fmn} F_{m,n} = m!
\displaystyle\sum_{\substack{a_1,\ldots a_l\\\sum a_i = n+k -1}}
\left(\frac{(k-1)!n!}{(a_1!\ldots a_l!)^2}
\displaystyle\sum_{\substack{k_1, \ldots k_l, k_i \le a_i\\ \sum k_i = k-1}}
\prod_{i=1}^l \binom{a_i}{k_i}\right)\displaystyle \prod_{i=1}^{l} w_i^{2a_i}\,.
\end{equation}
For         notational         convenience,         we         define
\begin{equation}
\label{eq:comb}
f(a_1, \ldots, a_l) = \frac{(l-3)!n!}{(a_1!a_2!\ldots a_l!)^2}
\displaystyle\sum_{\substack{k_1, \ldots k_l, k_i \le a_i\\ \sum k_i = k-1}}
\ \prod_{i=1}^l \binom{a_i}{k_i}\,;
\end{equation}
where $l=k+2$, so that
$$
F_{m,n} = m!
\displaystyle\sum_{\substack{a_1,\ldots a_l\\ \sum a_i = n+k -1}}
f(a_1,\ldots, a_l)\displaystyle \prod_{i=1}^{l} w_i^{2a_i}\,.
$$
Multiplying  and  dividing~\eqref{eq:comb}  by the factor $(n+l-3)!,$
and   taking   into  consideration  that  $n+l-3=n+k-1=\sum  a_i$  we
rewrite~\eqref{eq:comb} as
\begin{equation}
\label{eq:comb2}
f(a_1, \ldots, a_l) = \frac{1}{a_1!\ldots a_k!}\frac{\binom{n+l-3}{a_1, \ldots a_l}}{\binom{n+l-3}{n}} \sum_{k_1, \ldots k_l}
\prod_{i=1}^l \binom{a_i}{k_i}\,.
\end{equation}
We have
\begin{equation}
\label{eq:binom}
\sum_{k_1, \ldots k_l}
\prod_{i=1}^l \binom{a_i}{k_i} = {\binom{n+l-3}{n}}
\end{equation}
by  a  classical  combinatorial  argument. Indeed, $\binom{n+l-3}{n}$
represents  the  ways to select a subset of $l-3$ elements from a set
of  size  $n+l-3$. On the other hand, suppose the set of size $n+l-3$
contained elements of $l$ distinct types. To pick $l-3$ elements, one
can  choose $k_1$ of the first kind, up to $k_l$of the $l^{th}$ kind,
with  $\sum  k_i  =  l-3$.  There are $\prod\binom{a_i}{k_i}$ ways of
doing  this.  Summing  over all possible $k_1, \ldots k_l$ with $\sum
k_i    =    l-3$,    we    obtain   $\binom{n+l-3}{n}$.   Simplifying
(\ref{eq:comb2}) using~\eqref{eq:binom}, we obtain
\begin{equation}
\label{eq:coeff:final}
f(a_1, \ldots, a_l) =
\frac{\binom{n+l-3}{a_1, \ldots a_l}}{a_1!\ldots a_k!} =
\frac{1}{(n+l-3)!} \binom{n+l-3}{a_1, \ldots a_l}^2\,.
\end{equation}
Noting that $n+l-3 = (m+n)/2 -1$, we obtain~\eqref{eq:local:poly}.

In the cases $m=n$ and $m=n-2$ a similar argument applied to the base
polynomials $F_{1,1} (w_1, w_2) = 1$ and $F_{0,2}(w) = 1$ yields:
\begin{equation}
F_{m,n}=\begin{cases}
\displaystyle m \sum_{i=0}^{m-1}\binom{m-1}{i}^2
w_1^{2i}w_2^{2(m-1-i)} & m=n\\ \\
w^{2m} & m = n-2,
\end{cases}
\end{equation}
\end{proof}

\noindent\textbf{Values of $\pmb{F_{m,n}}$ for small $\pmb{m,n}$}.
To  illustrate  the above theorem, we compute the values of $F_{m,n}$
which      are     involved     in     volume     calculations     of
$\cQ(1^\noz,-1^{\noz+4})$     for     $\noz=1,2$     performed     in
\S\ref{subsec:1:5} and \S\ref{subsec:2:6}.
\medskip

\begin{table}[ht]\caption{Values of $F_{m,n}$ for $(m,n)$ small.}

\hspace*{-90pt}
\begin{minipage}[l]{0.10\textwidth}
\begin{tabular}{|c|c|}
\hline
\multicolumn{2}{|c|}{}\\
\multicolumn{2}{|c|}{\text{Valence } 1}\\
\hline&\\
$m,n$  & $F_{m,n}$\\
\hline&\\
$0,2$     & $1$    \\
\hline &\\
$1,3$    & $w^2$  \\
\hline &\\
$2,4$  & $w^4$  \\
\hline&\\
$3,5$  &$ w^6$  \\ \hline
\end{tabular}
\end{minipage}
\hspace{1cm}
\begin{minipage}[c]{0.10\textwidth}
\begin{tabular}{|c|c|}
\hline
\multicolumn{2}{|c|}{}\\
\multicolumn{2}{|c|}{\text{Valence } 2}\\
\hline&\\
$m,n$ &  $F_{m,n}$\\
\hline&\\
$1,1$ & $1$\\
\hline &\\
$2,2$ &  $2(w_1^2+w_2^2)$\\
\hline &\\
$3,3$ & $3(w_1^4 +  4 w_1^2 w_2^2  +  w_2^4)$\\
\hline
\end{tabular}
\end{minipage}
\hspace{3.5cm}
\begin{minipage}[r]{0.05\textwidth}
\begin{tabular}{|c|c|}
\hline
\multicolumn{2}{|c|}{}\\
\multicolumn{2}{|c|}{\text{Valence } 3}\\
\hline&\\
$m,n$ & $F_{m,n}$\\
\hline&\\
$2,0$ &  $2$\\
\hline&\\
$3,1$&$ 6(w_1^2+w_2^2+w_3^2)$\\
\hline
\end{tabular}
\end{minipage}
\end{table}

%

\subsection{Total Sums}
\label{sec:subsec:total}
We first recall the following standard fact:
\begin{lemma}
\label{lemma:zeta:functions}
As $\Deg \to \infty$,
\begin{displaymath}
\sum_{\substack{h \in \N^k,\ w \in \N^k\\
h \cdot w \le \Deg}}
w_1^{a_1+1} \dots w_k^{a_k+1} \sim
\frac{\Deg^{a+2k}}{(a+2k)!}\cdot
\prod_{i=1}^k (a_i+1)! \, \zeta(a_i+2)\,,
\end{displaymath}
where $a_i\in\N$ for $i=1,\dots,k$ and $a=a_1+\dots+a_k$.
\end{lemma}
\begin{proof}
Denote  by $\Delta^k$ the simplex $x_1+\dots+x_k\le 1$ in $\R_+^{k}$.
Introducing the variables $x_i:=\frac{w_i h_i}{\Deg}$ we can approximate
the initial sum by the following sum of integrals:
\begin{multline*}
\sum_{h\cdot w\le \Deg} w_1^{a_1+1} \dots w_k^{a_k+1} \sim
\\
\sum_{h\in\N^k}\int_{\Delta^k}
\left(\frac{x_1 \Deg}{h_1}\right)^{a_1+1}\hspace*{-6pt}\dots\left(\frac{x_k \Deg}{h_k}\right)^{a_k+1}
\left(\frac{\Deg}{h_1}dx_1\right)\dots\left(\frac{\Deg}{h_k}dx_k\right)=\\
\Deg^{a+2k}\cdot
\int_{\Delta^k} x_1^{a_1+1}\dots x_k^{a_k+1}\,dx_1\dots dx_k\,\cdot
\sum_{h\in\N^k} \frac{1}{h_1^{a_1+2}}\dots\frac{1}{h_k^{a_k+2}}\,.
\end{multline*}
It remains to note that
$$
\int_{\Delta^k} x_1^{a_1+1}\dots x_k^{a_k+1}\,dx_1\dots dx_k\,=
\frac{(a_1+1)!\dots\, (a_k+1)!}{(a+2k)!}\,.
$$
\end{proof}

The  calculations  of  the  local  polynomials  allow us to obtain an
expression for the number of connected pillowcase covers of degree at
most  $\Deg$  having the ramification points only over the corners of
the   pillowcase   and  having  the  following  ramification  profile
(indicating  the  total  number  of ramification points over the four
corners  of  the  pillowcase  together). The cover has exactly $\noz$
ramification  points  of degree $3$, $\noz+4$ nonramified points; all
remaining  points over the corners of the pillowcase have degree $2$.
Imposing this ramification profile and connectedness of $\hat \cP$ is
equivalent to requiring that $\hat\cP\in\cQ(1^\noz,-1^{\noz+4})$.

As  explained  in \S\ref{sec:subsec:jenkins}, see also
Figure~\ref{fig:global:graph:and:layers}, every such pillowcase cover
defines a ``global tree'' $\Tree$ which edges correspond to cylinders
filled  with  horizontal  periodic  trajectories,  and whose vertices
correspond  to  ``singular  layers''.  We  stress  that a global tree
represents  only  the adjacency of the cylinders to the same singular
layers,  and  have  almost  nothing  in common with the ribbon graphs
considered   in   \S\ref{sec:subsec:jenkins};  numerous  ribbon
graphs might be hidden behind a vertex of the global tree.

Let some horizontal singular layer $v$ contain $m_v$ zeroes and $n_v$
simple  poles.  The  valence $\va_v$ of the vertex of the global tree
$\Tree$   represents   the   number  of  cylinders  adjacent  to  the
corresponding  layer.  In  other  words,  it stands for the number of
boundary  components  of  the ribbon graph corresponding to the layer
(``faces''  in  terminology  of \S\ref{sec:subsec:jenkins}). We
have   seen  in  \S\ref{sec:subsec:jenkins}  that  the  valence
$\va_v$  and the degree $2a_v:=\operatorname{deg} F_{m_v,n_v}$ of the
corresponding local polynomial are related to $m_v$ and $n_v$ as
\begin{equation}
\label{eq:m:n:power:valence}
\left\{\begin{array}{lll}
a_v&=&\cfrac{m_v+n_v}{2}-1\\[-\halfbls]\\
\va_v&=&\cfrac{m_v-n_v}{2}+2
\end{array}
\right.
\qquad\qquad\qquad
\left\{\begin{array}{lll}
m_v&=&a_v+\va_v-1\\
n_v&=&a_v-\va_v+3
\end{array}
\right.\,.
\end{equation}
Since the number $n_v$ is nonnegative, the degree $2a_v$ and the
valence $\va_v$ satisfy the relation
\begin{equation}
\label{eq:power:ge:valence:minus:3}
a_v\ge \va_v-3\quad\text{for any vertex }v\in\Tree\,.
\end{equation}
Also,  since  the  total  number  of  zeroes  and poles is $2\noz+4$,
summing  up  the  expression  for $a_v$ over all vertices of the tree
$\Tree$, we get
\begin{equation}
\label{eq:sum:a}
\sum_{v\in\Tree} a_v = \noz+2-|V(\Tree)|\,,
\end{equation}
where $|V(\Tree)|$ denotes the number of vertices in $\Tree$.

Reciprocally,  given  any  connected  tree  $\Tree$ with at least two
vertices, and any integer $\noz$ satisfying
\begin{equation}
\label{eq:number:of:vertices}
|V(\Tree)|\le\noz+2\,,
\end{equation}
consider  any partition  of the number $\noz+2-|V|$ into nonnegative
integers
$$
a=a_{v_1}+\dots+a_{v_{|V|}}\,,
$$
where  elements $a_v$ of the partition are enumerated by the vertices
of   the   tree   $\Tree$.   If   for   every   $v$  in  $\Tree$  the
inequality~\eqref{eq:power:ge:valence:minus:3}                 holds,
equations~\eqref{eq:m:n:power:valence}  uniquely  determine for every
vertex  $v$ a couple of nonnegative integers $n_v, m_v$ which are not
simultaneously equal to zero. By construction,
$$
\sum_{v\in\Tree} m_v = \noz \qquad\text{and}\qquad \sum_{v\in\Tree} n_v = \noz+4\,.
$$

\begin{definition}
\label{def:decorated}
Given an integer $\noz\in\N$ and a tree $\Tree$ with at least two and at most
$\noz+2$ vertices, by \textit{decoration} of the tree $\Tree$
we call a partition
$
a=a_{v_1}+\dots+a_{v_{|V|}}\,,
$
enumerated   by   the   vertices   of   the   tree   and   satisfying
relations~\eqref{eq:power:ge:valence:minus:3} and~\eqref{eq:sum:a}.
\end{definition}

We  have  just  proved the following Lemma.
\begin{lemma}
A global tree of any pillowcase cover in $\cQ(1^\noz,-1^{\noz+4})$ is
naturally      \textit{decorated}      in      the      sense      of
Definition~\ref{def:decorated}.        Any       decorated       tree
satisfying~\eqref{eq:sum:a}  corresponds  to  some  actual pillowcase
cover in $\cQ(1^\noz,-1^{\noz+4})$.
\end{lemma}

Now  we  are  ready  to  count  the  number  of  pillowcase covers in
$\cQ(1^\noz,-1^{\noz+4})$  of  degree at most $\Deg$ represented by a
given  decorated  tree $(\Tree,a)$. Let $V$ be the set of vertices of
the  tree  $\Tree$,  let  $E$  be  the set of edges of $\Tree$. Since
$\Tree$  is  a  tree  we  have $|E|=|V|-1$. We always assume that the
labellings   of   vertices   and   edges,  that  is,  the  bijections
$V\to\{1,\dots, |V|\}$ and $E\to\{1,\dots, |E|\}$ are fixed.

Recall  that  to  each  edge  $e_j$ of $\Tree$ we associate a pair of
variables   $h_j$  and  $w_j$  which  represent  the  height  of  the
corresponding cylinder and its width (length of the waist curve). The
decoration  associates  a  pair  of nonnegative integers $m_i,n_i$ to
each vertex $v_i$ of the tree; $m_i,n_i$ are not simultaneously equal
to  zero.  We  associate  to  every  vertex  $v$ the local polynomial
$F_{m_i,n_i}(w_{j_1},\dots,w_{j_{\va(v)}})$  where  $\va(v)$  is  the
valence of the vertex $v=v_i$, and indices $\{j_1,\dots,j_{\va(v)}\}$
enumerate the edges $e_{j_1},\dots,e_{j_{\va(v)}}$ adjacent to $v_i$.

Let   $|\operatorname{Aut}(\Tree,a)|$   be  the  cardinality  of  the
automorphism  group  of  the  decorated  tree  $(\Tree,a)$,  and  let
$k:=|E|$  be  the number of the edges of the tree $\Tree$. The number
of  ways to \textit{give names} to $m_i$ zeroes and to $n_i$ poles at
the layer $v_i$, where $i=1,\dots,k+1$, equals
$$
\frac{1}{|\operatorname{Aut}(\Tree,a)|}
\binom{m}{m_1,\dots,m_{k+1}} \binom{n}{n_1,\dots,n_{k+1}}\,, \qquad k=|E|\,.
$$
Hence,   by  Lemma~\ref{lm:number:of:pillowcase:covers:of:given:type}
the   number   of   pillowcase   covers  of  degree  at  most  $\Deg$
corresponding  to  the  decorated  tree  $(\Tree,a)$  is equal to the
following sum
\begin{equation}
\label{eq:contribution}
\sum_{\substack{h \in \N^k,\ w \in \N^k\\
h \cdot w \le \Deg}}
\frac{1}{|\operatorname{Aut}(\Tree,a)|}
\binom{m}{m_1,\dots,m_{k+1}} \binom{n}{n_1,\dots,n_{k+1}}
(2w_1) \ldots (2w_k)\prod_{i=1}^k F_{m_i,n_i}\,,
\end{equation}
where  the  arguments  of $F_{m_i,n_i}(w_{j_1},\dots,w_{j_{\va(v)}})$
correspond  to  edges  $e_{j_1},\dots,e_{j_{\va(v)}}$ adjacent to the
vertex  $v_i$.  Note  that  the definition of the decoration, and the
construction  of the local polynomials $F_{m_i,n_i}$ implies that any
monomial  in  $w_1,\dots,w_k$ of the above sum has total degree equal
to $\dim_{\C{}}\cQ(1^K,-1^{K+4})=2K+2$.

Define the    formal    operation
$$
\cZ\ :\quad
\prod_{i=1}^{k} w_i^{b_i+1} \longmapsto
\frac{2}{(b+2k-1)!} \prod_{i=1}^{k} \big((b_i+1)!\cdot \zeta(b_i+2)\big)\,,
$$
where  $  b = \sum b_i$. For $b+2k=\dim_{\C{}}\cQ(1^K,-1^{K+4})$ this
operation  corresponds  to  the  following sequence of operations. We
first  apply  Lemma~\ref{lemma:zeta:functions}  to the sum $\sum_{h.w
\le      \Deg}     \prod_{i=1}^{k}     w_i^{b_i+1}$     to     obtain
$\cfrac{N^{b+2k}}{(b+2k)!}\prod_{i=1}^{k}  (b_i+1)!\,  \zeta(b_i  + 2)$.
Then,                                                       following
Lemma~\ref{lm:number:of:pillowcase:covers:of:given:type}   we  divide
the  resulting sum by $N^{\dim_{\C{}}\cQ}$ and multiply the result by
$2\dim_{\C{}}\cQ$.

Summing  up  the  contributions~\eqref{eq:contribution} of individual
decorated                       trees,                       applying
Lemmas~\ref{lm:number:of:pillowcase:covers:of:given:type}
and~\ref{lemma:zeta:functions},  and  using  the  notation  $\cZ$  we
obtain $\Vol\cQ(1^K,-1^{K-4})$. On the other hand, by Theorem for the
volumes    $\Vol\cQ(1^K,-1^{K+4})$    stated    at    the    end   of
\S\ref{subsec:counting:pillowcase:covers}
$$
\Vol \cQ_1\left(1^K, -1^{K+4}\right) =
\frac{\pi^{2K+2}}{2^{K-1}}\,.
$$
Comparing the two expressions for the volume, we obtain the following
identity

\begin{theorem}
\label{th:volume:through:pillows}
For any $K\in\N{}$ the following identity holds:
\begin{multline}
\label{eq:pillows:volume}
\frac{\pi^{2K+2}}{2^{K-1}}=
\sum_{k=1}^{K+1}
\sum_{\substack{\text{Connected}\\\text{trees }\Tree\\\text{with }k\text{ edges}}}
\sum_{\substack{\text{Admissible}\\\text{decorations}\\\mathbf{a}\text{ of }\Tree}}
\\
\frac{2^k}{|\operatorname{Aut}(\Tree,a)|}\cdot
\binom{m}{m_1,\dots,m_{k+1}} \binom{n}{n_1,\dots,n_{k+1}}\cdot
\cZ\!\left(w_1 \ldots w_k\prod_{i=1}^k F_{m_i,n_i}\right)\,.
\end{multline}
\end{theorem}

Below, we illustrate the volume calculations for the strata
$\cQ_1(1^1, (-1)^5)$ and $\cQ_1(1^2, (-1)^6)$.

\subsection{Stratum $\pmb{\cQ(1^1,-1^5)}$}\label{subsec:1:5}

%
Let

$$
c(\Tree,a):=\frac{1}{|\operatorname{Aut}(\Tree,a)|}\cdot
\binom{m}{m_1,\dots,m_{k+1}} \binom{n}{n_1,\dots,n_{k+1}}\,.
$$
For  $K=1$  there  are  only  two trees with at least $2$ and at most
$K+2$ vertices. Each of these trees admits a unique decoration. Thus,
the  sum in the right-hand side of~\eqref{eq:pillows:volume} contains
only two summands described in the table below.

$$
\begin{array}{|c|c|c|c|}
\hline
\text{Tree} &  {\displaystyle\prod_{i=1}^k F_{m_i,n_i}}& c(\Tree,a) &\text{Contribution}\\
\hline
\multicolumn{4}{c}{}\\
\multicolumn{4}{c}{k=1 \text{ cylinder}}\\
\hline&&&\\
\begin{picture}(10,0)(0,-5)
\put(0,0){\circle{4}}
\put(5,0){\tiny 1,3}
\put(0,-2){\line(0,-1){26}}
\put(0,-30){\circle{4}}
\put(5,-32){\tiny 0,2}
\end{picture}
& F_{1,3}(w_1) \cdot F_{0,2}(w_1)=
& 1\cdot \binom{1}{1,0}\binom{5}{3, 2}& 40\cdot\zeta(4)=
\\
&
=w_1^2 \cdot 1
&&=\cfrac{4}{9}\cdot\pi^4\\
&&&\\
\hline
\multicolumn{4}{c}{}\\
\multicolumn{4}{c}{k=2 \text{ cylinders}}\\
\hline&&&\\
\begin{picture}(10,35)(0,-30)
\put(0,0){\circle{4}}
\put(5,0){\tiny 0,2}
\put(0,-2){\line(0,-1){26}}
\put(0,-30){\circle{4}}
\put(5,-32){\tiny 1,1}
\put(0,-32){\line(0,-1){26}}
\put(0,-60){\circle{4}}
\put(5,-62){\tiny 0,2}
\end{picture}
& F_{0,2}(w_1)\cdot F_{1,1}(w_1,w_2) \cdot F_{0,2}(w_2)=
& \frac{1}{2}\cdot\binom{1}{1, 0, 1}\binom{5}{2, 1, 2}&20\cdot\zeta^2(2)=\\
&
=1\cdot 1\cdot 1
&&\cfrac{5}{9}\cdot\pi^4\\
&&&\\
\hline
\end{array}
$$
\bigskip

\noindent Adding the two terms we get the following value for the volume (recall all zeroes and
poles are numbered):
$$
\Vol\cQ_1(1,-1^5) = \pi^4
$$
Up to the factor $120=5!$ coming from enumeration of simple poles
it matches the value
$$
\Vol\cH_1(2)=\cfrac{\pi^4}{120}
$$
from~\cite{Eskin:Masur:Zorich},  where $\cH_1(2)$ denotes the stratum
of  unit-area  Abelian  differentials with one double zero on a genus
$2$  surface.  The  two  values  agree  since we have the isomorphism
$\cH(2)  \cong \cQ(1,-1^5)$ by taking the quotient of each surface in
$\cH(2)$ by the hyperelliptic involution.

\subsection{Stratum $\pmb{\cQ(1^2,-1^6)}$}\label{subsec:2:6}

%
   %
For $K=2$ the tree can contain from one to three edges; corresponding
decorated  trees  and  their  contributions  to  the  right-hand side
of~\eqref{eq:pillows:volume} are presented in the table below.

$$
\begin{array}{|c|c|c|c|}
\hline
\text{Tree} &  {\displaystyle\prod_{i=1}^k F_{m_i,n_i}}& c(\Tree,a) &\text{Contribution}\\
\hline
\multicolumn{4}{c}{}\\
\multicolumn{4}{c}{k=1 \text{ cylinder}}\\
\hline&&&\\
\begin{picture}(10,0)(0,-5)
\put(0,0){\circle{4}}
\put(5,0){\tiny 2,4}
\put(0,-2){\line(0,-1){26}}
\put(0,-30){\circle{4}}
\put(5,-32){\tiny 0,2}
\end{picture}
& F_{2,4}(w_1) \cdot F_{0,2}(w_1)=
&1\cdot \binom{2}{2, 0}\binom{6}{4, 2}&60\cdot\zeta(6)=
\\
&
=w_1^4\cdot 1
&
&\cfrac{4}{63}\cdot\pi^6
\\
&&&\\
\hline
&&&\\
\begin{picture}(10,0)(0,-5)
\put(0,0){\circle{4}}
\put(5,0){\tiny 1,3}
\put(0,-2){\line(0,-1){26}}
\put(0,-30){\circle{4}}
\put(5,-32){\tiny 1,3}
\end{picture}
& F_{1,3}(w_1) \cdot F_{1,3}(w_1)=
& \frac{1}{2}\cdot\binom{2}{1,1}\binom{6}{3,3}&80\cdot\zeta(6)=
\\
&
=w_1^2\cdot w_1^2
&
&\cfrac{16}{189}\cdot\pi^6\\
&&&\\
\hline\multicolumn{3}{r}{\text{Subtotal:}}
&\multicolumn{1}{c}{\cfrac{4}{27}\cdot\pi^6}\\
\multicolumn{4}{c}{}\\
\multicolumn{4}{c}{k=2 \text{ cylinders}}\\
\hline&&&\\
\begin{picture}(10,35)(0,-30)
\put(0,0){\circle{4}}
\put(5,0){\tiny 0,2}
\put(0,-2){\line(0,-1){26}}
\put(0,-30){\circle{4}}
\put(5,-32){\tiny 2,2}
\put(0,-32){\line(0,-1){26}}
\put(0,-60){\circle{4}}
\put(5,-62){\tiny 0,2}
\end{picture}
& F_{0,2}(w_1) \cdot F_{2,2}(w_1,w_2) \cdot F_{0,2}(w_2)=
& \frac{1}{2}\cdot\binom{2}{0,2,0}\binom{6}{2, 2, 2}&72\cdot\zeta(2)\cdot\zeta(4)=
\\
&
=1\cdot 2(w_1^2+w_2^2)\cdot 1
&
&\cfrac{2}{15}\cdot\pi^6
\\
&&&\\
\hline
&&&\\
\begin{picture}(10,35)(0,-30)
\put(0,0){\circle{4}}
\put(5,0){\tiny 1,3}
\put(0,-2){\line(0,-1){26}}
\put(0,-30){\circle{4}}
\put(5,-32){\tiny 1,1}
\put(0,-32){\line(0,-1){26}}
\put(0,-60){\circle{4}}
\put(5,-62){\tiny 0,2}
\end{picture}
& F_{1,3}(w_1) \cdot F_{1,1}(w_1,w_2)\cdot F_{0,2}(w_2)=
&1\cdot\binom{2}{1, 1,0}\binom{6}{3, 1, 2}&48\cdot\zeta(2)\cdot\zeta(4)=
\\
&
=w_1^2\cdot 1\cdot 1
&
&\cfrac{4}{45}\cdot\pi^6\\
&&&\\
\hline\multicolumn{3}{r}{\text{Subtotal:}}
&\multicolumn{1}{c}{\cfrac{2}{9}\cdot\pi^6}
\end{array}
$$

$$
\begin{array}{|c|c|c|c|}
\hline
\text{Tree} &  {\displaystyle\prod_{i=1}^k F_{m_i,n_i}}& c(\Tree,a) &\text{Contribution}\\
\hline
\multicolumn{4}{c}{}\\
\multicolumn{4}{c}{k=3 \text{ cylinders}}\\
\hline&&&\\
\begin{picture}(10,35)(0,-30)
\put(0,0){\circle{4}}
\put(5,0){\tiny 0,2}
\put(0,-2){\line(0,-1){26}}
\put(0,-30){\circle{4}}
\put(5,-32){\tiny 1,1}
\put(0,-32){\line(0,-1){26}}
\put(0,-60){\circle{4}}
\put(5,-62){\tiny 1,1}
\put(0,-62){\line(0,-1){26}}
\put(0,-90){\circle{4}}
\put(5,-92){\tiny 0,2}
\end{picture}
& F_{0,2}(w_1)\cdot  F_{1,1}(w_1,w_2)\cdot
& \frac{1}{2}\cdot\binom{2}{0, 1, 1, 0}\binom{6}{2, 1, 1, 2}&24\cdot\zeta^3(2)=
\\
&
\cdot F_{1,1}(w_2,w_3)\cdot F_{0,2}(w_3)=
&
&\cfrac{1}{9}\cdot\pi^6
\\
&=1\cdot 1\cdot 1 \cdot 1&&\\
&&&\\
&&&\\
&&&\\
\hline
&&&\\
\begin{picture}(70,15)(-35,-10)
\put(0,0){\circle{4}}
\put(5,0){\tiny 0,2}
\put(0,-2){\line(0,-1){26}}
\put(0,-30){\circle{4}}
\put(5,-32){\tiny 2,0}
\put(-1.8,-31){\line(-2,-1){26}}
\put(-29.6,-44.9){\circle{4}}
\put(-35,-54){\tiny 0,2}
\put(1.8,-31){\line(2,-1){26}}
\put(29.6,-44.9){\circle{4}}
\put(25,-54){\tiny 0,2}
\end{picture}
& F_{0,2}(w_1)\cdot F_{2,0}(w_1,w_2,w_3)\cdot
& \frac{1}{6}\cdot
\binom{2}{0, 0, 0, 2}\binom{6}{2,2,2, 0}&4\cdot\zeta^3(2)=
\\
&
\cdot F_{0,2}(w_2)\cdot F_{0,2}(w_3)=
&
&\cfrac{1}{54}\cdot\pi^6\\
&=1\cdot 1\cdot 1\cdot 1&&\\
&&&\\
\hline\multicolumn{3}{r}{\text{Subtotal:}}
&\multicolumn{1}{c}{\cfrac{7}{54}\cdot\pi^6}
\end{array}
$$
\smallskip

Taking          the          total         sum         we         get
$\Vol\cQ(1^2,-1^6)=\left(\cfrac{4}{27}+\cfrac{2}{9}+\cfrac{7}{54}\right)\pi^6=
\cfrac{\pi^6}{2}$.


\end{document}